\documentclass[10pt]{amsproc}

\usepackage[all]{xy}

\usepackage[initials]{amsrefs}
\usepackage{amssymb}
\usepackage{setspace}
\usepackage{array}

\newtheorem{theorem}{Theorem}[section]
\newtheorem{lemma}[theorem]{Lemma}
\newtheorem{proposition}[theorem]{Proposition}
\newtheorem{corollary}[theorem]{Corollary}
\newtheorem{definition}[theorem]{Definition}

\newtheorem{setup}[theorem]{Setup}

\newtheorem{remark}[theorem]{Remark}

\newcommand{\bG}{{\mathbf G}}
\newcommand{\bB}{{\mathbf B}}
\newcommand{\cH}{{\mathcal H}}
\newcommand{\cF}{\mathcal F}

\newcommand{\F}{{\mathbb F}}
\newcommand{\R}{{\mathbb R}}

\newcommand{\C}{{\mathbb C}}
\newcommand{\Z}{{\mathbb Z}}
\newcommand{\N}{{\mathbb N}}
\newcommand{\bs}{\backslash}
\newcommand{\tr}{\operatorname{tr}}
\newcommand{\id}{\operatorname{id}}
\newcommand{\diag}{\operatorname{diag}}
\newcommand{\sub}{\operatorname{Sub}}
\newcommand{\irs}{\operatorname{IRS}}
\newcommand{\covol}{\operatorname{covol}}
\newcommand{\tor}{\operatorname{Tor}}
\newcommand{\betti}{\beta^{(2)}}

\numberwithin{equation}{section}

\title[$L^2$-Betti numbers of totally disconnected groups]{$L^2$-Betti numbers of totally disconnected groups and their approximation by Betti numbers of lattices}

\author{Henrik Densing Petersen}
\address{H.D.P., EGG, EPFL, CH-1015 Lausanne, Switzerland}
\email{henrik.petersen@epfl.ch}

\author{Roman Sauer}
\address{R.S., IAG, AG Topologie, KIT, 76131 Karlsruhe, Germany }
\email{roman.sauer@kit.edu}

\author{Andreas Thom}
\address{A.T., Institut f\"ur Geometrie, TU Dresden, 01062 Dresden, Germany}
\email{andreas.thom@tu-dresden.de}

\begin{document}

\onehalfspace

\begin{abstract}
The main result is a general approximation theorem for normalised Betti numbers for Farber sequences of lattices in totally disconnected groups. Further, we contribute to the general theory of $L^2$-Betti numbers of totally disconnected groups and provide exact computations of the $L^2$-Betti numbers of the Neretin group and Chevalley groups over the field of Laurent series over a finite field and their lattices. 
\end{abstract}

\maketitle
\setcounter{tocdepth}{1}
\tableofcontents

\section{Introduction}

The study of asymptotics of Betti numbers of sequences of lattices with increasing covolume in locally compact groups has a long history. In this article, we prove a general convergence result that holds for any Farber chain in any unimodular, totally disconnected group. Moreover, we show that the limit can be identified in a natural way with an $L^2$-invariant associated with the locally compact group. 

The first instances of such $L^2$-invariants for non-discrete locally compact groups 
are Gaboriau's first $L^2$-Betti number of a unimodular graph~\cite{MR2221157}, which is essentially one of its automorphism group, and the $L^2$-Betti numbers of buildings in the work of  
Dymara~\cite{dymara} and Davis-Dymara-Januszkiewicz-Okun~\cites{davisetal}, which are essentially ones of the automorphism groups of the buildings. 

 In his PhD thesis~\cite{Pe11} the first named author developed a homological theory of $L^2$-Betti numbers that applies to all unimodular, locally compact groups and generalises the one by 
L\"uck~\cite{lueck-book}*{Chapter~6} for the discrete case, see also \cites{MR123, MR3158720}. Computations reduce to the discrete case whenever the locally compact group~$G$ possesses a lattice~$\Gamma$ as in this case the relation between the $L^2$-Betti numbers 
is 
\begin{equation}\label{eq: lattice and G l2 Betti}
\beta_{k}^{(2)}(G,\nu) = \frac{\beta_{k}^{(2)}(\Gamma)}{{\rm covol}_{\nu}(\Gamma)}.
\end{equation}
This result is proved by Kyed-Petersen-Vaes~\cite{kyed+petersen+vaes}, the difficult case being non-uniform lattices. Ultimately, it relied on Gaboriau's deep work on $L^2$-invariants for equivalence relations \cite{MR1953191}. We present a short proof of Equation \eqref{eq: lattice and G l2 Betti} for arbitrary lattices in totally disconnected groups in Section~\ref{sub: l2 betti and their lattices}. This also gives a direct proof of proportionality of $L^2$-Betti numbers of lattices in the same locally compact, totally disconnected group which was obtained (for all locally compact groups) by Gaboriau~\cite{MR1953191}*{Corollaire 0.2}. We also note that the vanishing of $L^2$-Betti numbers of locally compact second countable groups is a coarse invariant~\cite{sauer+schroedl}. 

A fundamental result in the theory of $L^2$-Betti numbers is L\"uck's approximation theorem~\cite{Lu94} which expresses $L^2$-Betti numbers as limits of normalised Betti numbers along a  residual chain of finite index normal subgroups. Farber relaxed the condition of normality for the sequence of finite index subgroups in the approximation theorem to what is now called a  \emph{Farber sequence}~\cite{farber}. The relevance of this notion was explained in the context of invariant random subgroups by Bergeron and Gaboriau in \cite{MR2059438}. The notion of Farber sequence admits a probabilistic generalisation to the situation where the discrete group is replaced by a locally compact totally disconnected group and the chain of finite index subgroups by a sequence of lattices. 

	Throughout the entire article, let $G$ denote a locally compact totally disconnected, second countable, unimodular group with Haar measure~$\nu$ unless explicitly stated otherwise.  We will denote by $\mathcal{K}(G)$ the ordered set of compact-open subgroups of $G$.

\begin{definition}\label{defn: farber condition}
		A sequence $(\Gamma_i)_{i\in\N}$ of lattices in $G$ is a \emph{Farber sequence} if for every compact open subgroup $K<G$ and for every right coset $C$ of $K$ 
		the probability that a conjugate $g\Gamma_i g^{-1}$ meets $C\backslash\{e\}$ tends to zero as $i\to\infty$. 
\end{definition}

Here the probabilities are taken with respect to the probability Haar measures on $G/\Gamma_i$.  Note that one could replace \emph{for every compact open subgroup $K<G$} by 
\emph{for some compact open subgroup $K<G$} in the above definition. 

A lattice $\Gamma<G$ is called \emph{cocompact} or \emph{uniform} if $G/\Gamma$ is compact. Note that if $\Gamma$ is cocompact, then 
there is some compact open subgroup $K<G$ so that every conjugate 
of $\Gamma$ meets $K$ only in $\{e\}$. The following notion of uniform discreteness is a family version of 
cocompactness. The notion of weak uniform discreteness was introduced by Gelander~\cite{gelander}. 

\begin{definition} 
	Let $F$ be a family of lattices of~$G$. It is called \emph{uniformly discrete} if there is a neighbourhood $U\subset G$ of the identity so that 
	every conjugate of a lattice in $F$ meets $U$ only in~$\{e\}$. We call $F$ \emph{weakly uniformly discrete} if for every $\epsilon>0$ there is a neighbourhood $U\subset G$ of the identity  such that the probability that a conjugate of a lattice in $F$ intersects $U$ non-trivially is at most~$\epsilon$.   
\end{definition}  

The above definition works for every locally compact group $G$. If $G$ is also totally disconnected, one obtains an equivalent definition by replacing  \emph{neighborhood of the identity} with \emph{compact open subgroup}. 

Next we turn to our main result (proved in Subsection~\ref{sub: proof of main})  which provides an approximation theorem for Farber sequences of lattices in~$G$. A \emph{$G$-CW-complex} is a CW-complex with a $G$-action that permutes open cells such that for an open cell $e$ with 
$ge=e$ the map $x\mapsto gx$ for $x \in e$, is the identity (cf.~~\cite{tomdieck}*{II.1}).
A contractible $G$-CW-complex whose 
stabiliser groups are compact and open is called a \emph{topological model of $G$}. 
 A topological model for~$G$ always exists and is unique up to $G$-homotopy equivalence~\cite{lueck-survey}. We call a $G$-CW-complex cocompact if it consists of only finitely many $G$-orbits of cells.

\begin{theorem}\label{thm:main}
Assume that $G$ admits a topological model whose $(n+1)$-skeleton is cocompact. Let $(\Gamma_i)_{i\in \mathbb{N}}$ be a Farber sequence of lattices.  Then 
\begin{equation*}
\beta_n^{(2)}(G,\nu) \leq \liminf_{i \to \infty} \frac{\beta_n(\Gamma_i)}{\operatorname{covol}_{\nu}(\Gamma_i)} \quad\text{ for all $i\in\{0,\ldots, n\}$.}
\end{equation*}
If, in addition, $(\Gamma_i)_{i\in\N}$ is uniformly discrete, then
\begin{equation*} %\label{equnif}
\beta_n^{(2)}(G,\nu) = \lim_{i \to \infty} \frac{\beta_n(\Gamma_i)}{\operatorname{covol}_{\nu}(\Gamma_i)} \quad \text{ for all $i\in\{0,\ldots, n\}$.}
\end{equation*}
\end{theorem}

Note that the inequality in Theorem \ref{thm:main} is the {\it non-trivial} and more interesting inequality, when one follows the proof of L\"uck's approximation theorem in the discrete case. Indeed, it is usually the other inequality -- sometimes called Kazhdan's inequality -- which follows more easily. However, in our situation we were not able to prove Kazhdan's inequality without further assumptions.

\vspace{0.1cm}

Next we relate the theorem to the subject of invariant random subgroups which emerged in the last several years. 

The Chabauty topology (cf.~Appendix~\ref{sec: chabauty}) on the 
set $\sub_G$ of closed subgroups of~$G$ turns $\sub_G$ into a compact, metrisable space. The group $G$ acts continuously 
on $\sub_G$ by conjugation. The \emph{space of invariant random subgroups} $\irs_G$  of $G$ is the space of $G$-invariant Borel probability measures on $\sub_G$ endowed with the topology of weak convergence. Point measures $\delta_N$ of closed normal subgroups $N<G$ yield examples of invariant random subgroups. Another source of invariant random subgroups are 
lattices $\Gamma<G$. Indeed, the push forward of the normalised Haar measure on $G/\Gamma$ under the measurable map 
\[ G/\Gamma\rightarrow\sub_G, ~g\Gamma\mapsto g\Gamma g^{-1} \]
is an invariant random subgroup, denoted by $\mu_\Gamma\in \sub_G$. 
Statements about random conjugates of a lattice~$\Gamma$ always refer to the measure $\mu_\Gamma$. 

The relationship of convergence of invariant random subgroups and Farber chains is made precise in the following proposition. We thank Arie Levit for discussions about the relation between convergence in $\irs_G$ and Farber convergence, in particular for pointing out that convergence in $\irs_G$ and weak uniform discreteness imply the property of being a Farber sequence. 

\begin{proposition}[cf.~Appendix~\ref{sec: chabauty}]\label{prop: on Farber vs IRS}
If $(\Gamma_i)_{i\in \mathbb{N}}$ is a Farber sequence then the sequence $(\Gamma_i)_{i\in \mathbb{N}}$ converges to the trivial subgroup seen as an invariant random subgroup. The converse is true provided the sequence of lattices is weakly uniformly discrete.
\end{proposition}

\begin{remark}\label{rem: weak uniform discreteness}
Gelander~\cite{gelander} proves that 
the family of all lattices in a connected non-compact simple Lie group
is uniformly weakly discrete which implies the Kazhdan-Margulis theorem. 
The family of all lattices in a $p$-adic analytic group $G$ is uniformly 
discrete since $G$ contains an compact open torsionfree subgroup. 
\end{remark}

A pioneering study of invariant random subgroups in 
simple Lie groups has been done in~\cites{7sam_announce, 7sam} by Ab\'ert, Bergeron, Biringer, Gelander, Nikolov, Raimbault, and Samet. They prove structural results on the 
space $\irs_G$ for $G$ being a simple Lie group with property $(T)$ and a convergence result for uniformly discrete sequences of lattices in such $G$, similar to Theorem~\ref{thm:main}. We cannot expect interesting structural results on $\irs_G$ in the generality of arbitrary totally disconnected groups. However, for totally disconnected groups of algebraic origin this was achieved by Gelander and Levit~\cite{gelander+levit}. 

\begin{theorem}[Gelander-Levit] \label{ref: gelander+levit}
	Let $G$ be the $k$-points of a simple linear 
	algebraic group over a non-Archimedean local field~$k$. 
	Assume that~$G$ has property~(T). 
	Let $(\Gamma_i)_{i\in \mathbb{N}}$ be a sequence of lattices whose covolumes tend 
	to~$\infty$ as $i\to\infty$. If the characteristic of~$k$ is positive, we additionally assume that $(\Gamma_i)_{i\in \mathbb{N}}$ is uniformly discrete and torsionfree. Then $(\Gamma_i)_{i\in \mathbb{N}}$ is a Farber sequence. 
\end{theorem}

Thus, our results in combination of the work of Gelander and Levit nicely complement the results in \cite{7sam}, which are more analytic in spirit and only apply to Lie groups.

We discuss some specific examples of groups and applications in Section~\ref{sec: computations}. We obtain, in particular, new computations for Chevalley 
groups (see Theorem~\ref{thm: computation for lattices}):

\begin{theorem}\label{thm: computation for lattices in intro}
Let $\bG$ be a simple, simply connected Chevalley group over $\F_q$. 
Let $d$ be the rank of $\bG$. We normalise the Haar measure $\nu$ of $G$ so that the Iwahori group of $G$ has 
measure~$1$. Let $e_1,\ldots, e_d$ denote the exponents of the affine Weyl group of $\bG$. 
Then we have 
\[ \betti_d(\bG(\F_q((t^{-1}))), \nu)=
	\prod_{i=1}^{d} \frac{q^{e_i}-1}{1+q+\ldots +q^{e_i}}>0.
\]
All other $L^2$-Betti numbers vanish. The only non-vanishing $L^2$-Betti number 
of the lattice $\bG(\F_q[t])<\bG(\F_q((t^{-1})))$ is 
\[ \betti_d\bigl(\bG(\F_q[t])\bigr)=
	\prod_{i=1}^{d} (q^{e_i+1}-1)^{-1}.\]

If we assume, in addition, that $\bG$ is of classical type and $d>1$ or of type $\mathrm{\tilde E_6}$, and $q=p^a>9$ is a power of a prime $p>5$ if $\bG$ is not of type $\mathrm{\tilde A}_d$, then the only non-vanishing $L^2$-Betti number of an  arbitrary lattice $\Gamma<\bG(\F_q((t^{-1})))$ 
is in degree $d$ and satisfies 
\[ \betti_d(\Gamma)\ge \betti_d\left(\bG(\F_q[t])\right).\]
\end{theorem}

With the previous theorem we obtain an interesting application of Theorem~\ref{thm:main}: 

\begin{theorem}\label{thm: limit in intro}
We retain the setting of Theorem~\ref{thm: computation for lattices in intro}. 
Let $(\Gamma_i)_{i \in \mathbb N}$ be a Farber chain of lattices in $G$. Then 
\[ \prod_{i=1}^{d} \frac{q^{e_i}-1}{1+q+\ldots +q^{e_i}}\le\liminf_{i\to\infty} \frac{\beta_d(\Gamma_i)}{\covol_\nu(\Gamma_i)}. \]
\end{theorem}

Note that we cannot apply L\"uck's approximation theorem to prove Theorem~\ref{thm: limit in intro} -- even in the case of a residual chain -- since the groups $\Gamma_i$ do not satisfy the necessary finiteness condition: they are not of 
type $\mathrm{FP}_d$ (actually type $\mathrm{FP}_{d+1}$ would be needed) by a theorem of Bux and Wortman~\cite{bux}. The Betti numbers $\beta_d(\Gamma_i)$ are 
finite by a Theorem of Harder~\cite{harder}*{Satz~$1$}. 

We finish the paper 
with some miscellaneous observations, including a discussion of the Connes embedding problem for Hecke-von Neumann algebra which is a byproduct of our efforts to prove Theorem~\ref{thm:main}. 

\pagebreak
\section{Algebraic and homological aspects of Hecke algebras}

\subsection{The Hecke algebras of~$G$ and of a Hecke pair $(G,K)$}\label{sub: hecke pair}

Let $K< G$ be a compact open subgroup. The pair $(G,K)$ will be called a \emph{Hecke pair}. We normalise the Haar measure to $\nu(K)=1$.

The \emph{Hecke algebra} $\cH(G)$ of $G$ is the convolution algebra of compactly supported, locally constant, and complex-valued functions on~$G$. It is endowed with an involution given by $f^\ast(g)=\overline{f(g^{-1})}$. 
The \emph{Hecke algebra} $\cH(G,K)$ of the Hecke pair $(G,K)$ is the 
convolution algebra of compactly supported, $K$-bi-invariant, and complex-valued functions on $G$. 
Note that 
\begin{equation*}
\cH(G,K) = 1_K\cH(G)1_K = \quad\bigoplus_{{KgK \in K \bs G /K}} \C \cdot 1_{KgK},
\end{equation*}
where $1_{KgK}$ denotes the characteristic function of the double coset $KgK \in K \bs G / K$. 
Since $K \subset G$ is open, $K \bs G$ is discrete and by our normalisation the push-forward of the Haar measure is equal to the counting measure. Hence, the space of square integrable left $K$-invariant functions on $G$ is naturally identified with $\ell^2(K\bs G)$. We denote by $1_{Kg}$ the characteristic function of the set $Kg$. The set $\{1_{Kg} \mid Kg \in K \bs G \}$ is an orthonormal basis of $\ell^2(K\bs G)$. Convolution determines a natural representation 
\[ \pi\colon\cH(G,K)\to B(\ell^2(K\bs G))\]
of $\cH(G;K)$ as bounded operators on $\ell^2 (K\bs G)$. In the sequel let $[P]\in\{0,1\}$ denote the truth value of a mathematical expression~$P$. We have 
\begin{align*}
\pi(1_{KgK})(1_{Kh})=(1_{KgK} \ast 1_{Kh})(s) &= \int_{G} [t \in KgK]\cdot [t^{-1}s \in Kh] \ d\nu(t) \\
&= \int_{G} [t \in KgK]\cdot [s^{-1}t \in h^{-1}K] \ d\nu(t) \\
&= \int_{G} [t \in KgK]\cdot [t \in sh^{-1}K] \ d\nu(t) \\
&= \int_{G} [t \in KgK \cap sh^{-1}K] \ d\nu(t) \\
&= \nu(KgK \cap sh^{-1}K).
\end{align*}
Thus, setting $a_{g,h}^s := \nu(KgK \cap sh^{-1}K)$, we obtain
$$1_{KgK} \ast 1_{Kh} = \quad\sum_{{Ks \in K \bs G}} a_{g,h}^s \cdot 1_{Ks}.$$
The numbers $a_{g,h}^s$ are either $0$ or $1$ since $KgK \cap sh^{-1}K$ is either empty or equal $sh^{-1}K$. Since $K$ is compact, there are, for fixed $g,h \in G$, only finitely many $Ks \in K \bs G$ such that $a_{g,h}^s \neq 0$.

Similarly to the computation above, we see that
\begin{equation} \label{eq1} 1_{KgK} \ast 1_{KhK}=\quad \sum_{{KsK \in K \bs G/K}} b_{g,h}^s \cdot 1_{KsK}
\end{equation} with 
\[b_{g,h}^s := \nu(KgK \cap sKh^{-1}K).\] 
The numbers $b_{g,h}^s$ are non-negative integers since $KgK \cap sKh^{-1}K$ is a finite union of right cosets of $K$. They are called the \emph{structure constants} of the Hecke algebra $\cH(G,K)$.

\begin{definition}\label{def: trace on hecke algebra}
The linear functional 
\begin{equation*}
{\rm tr} \colon \cH(G,K) \to \C,\quad
{\rm tr}\bigl(~~\sum_{{KsK \in K \bs G/K}} a_{KsK} \cdot 1_{KsK} \bigr): = a_K.
\end{equation*}
defines a trace on the Hecke algebra $\cH(G,K)$. 
\end{definition}

Indeed, the trace property follows from unimodularity: 
\[{\rm tr}(1_{KgK} \ast 1_{KhK}) = \nu(KgK \cap Kh^{-1}K) = \nu(KhK \cap Kg^{-1}K) = {\rm tr}(1_{KhK} \ast 1_{KgK})\]
It is easy to see that ${\rm tr}$ is positive, unital and self-adjoint with respect to the natural involution $1_{KgK}^* = 1_{Kg^{-1}K}$. We conclude that $\cH(G,K)$ is a complex $*$-algebra with a unital, positive and faithful trace. Moreover, we obtain 
\[\tr(T) = \langle \pi(T)(1_K),1_K \rangle \quad \text{ for $T \in \cH(G,K)$.}\]

\subsection{The integral Hecke algebra} It is important for our approach that  homological invariants of the Hecke pair $(G,K)$ are defined not only in terms of its associated Hecke algebra, but in terms of some integral version of it -- a canonical $\Z$-algebra sitting inside $\cH(G,K)$ that plays the role of the integral group ring sitting inside the complex group.

We define the \emph{integral Hecke algebra} $\Z[K \bs G/K]$ to be the $\Z$-linear span of \[\{1_{KgK} \mid KgK \in K \bs G /K \} \subset \cH(G,K).\]
    By~\eqref{eq1}, $\Z[K \bs G/K]$ is closed under multiplication provided $\nu(K)=1$. For brevity, we will denote the integral Hecke algebra also by $\Z[G,K]$. The following lemma sheds some light on how to think of the integral Hecke algebra.

\begin{lemma}\label{lem: identification G-equivariant maps}
The action of $\Z[G,K]$ on $\Z[K \bs G]$ by convolution yields an identification of the ring $\Z[K \bs G/K]$ with  the ring \[\hom_{\Z}(\Z[K \bs G],\Z[K \bs G])^G\]
    of right $G$-equivariant endomorphisms of $\Z[K \bs G]$.
\end{lemma}
\begin{proof}
It is clear that there is a natural homomorphism $$\varphi \colon \Z[K \bs G/K] \to \hom_{\Z}(\Z[K \bs G],\Z[K \bs G])^{G},$$
which is given by convolution on the left. Since the value of $\varphi(a)$ on the trivial coset $1_K \in \Z[K \bs G]$ allows to recover $a \in \Z[K \bs G/K]$, the map $\varphi$ is injective. Moreover, by $G$-equivariance, any element  $$\alpha \in \hom_{\Z}(\Z[K \bs G],\Z[K \bs G])^{G}$$ is determined by its value $\alpha(1_K)$ on the trivial coset. Now, $\alpha(1_K) = \sum_{Ks \in K \bs G} a_{Ks} \cdot 1_{Ks}$. But $a_{Ks}$ is constant on double cosets $K \bs G /K$, since $\alpha$ is $G$-equivariant and $1_K$ is fixed by $K$. This shows that $$\alpha(1_K) = \sum_{KsK \in K \bs G/K} a_{KsK} \cdot 1_{KsK}$$ and proves the claim.
\end{proof}

Note that the augmentation homomorphism $\varepsilon \colon \Z[G , K] \to \Z$ which is given by $\varepsilon(1_{KgK}):= \nu(KgK)$ defines a $\Z[G , K]$-module structure on $\Z$.

\subsection{Homological algebra for the Hecke algebra}\label{sub: homological algebra}

Let $\mathcal{K}(G)$ denote the set of compact open subgroups of $G$ partially ordered by inclusion.

A (possibly non-unital) ring $R$ is called \emph{idempotented} if for every finite set of elements $S\subset R$ there 
is an idempotent $q\in R$ such that $qx=xq=x$ for all $x\in S$. A (left) $R$-module $M$ is \emph{non-degenerate} if 
$M=R\cdot M$. It is easy to see that the category of non-degenerate $R$-modules is an abelian category. For an idempotent $q\in R$ the $R$-module $Rq$ is projective in this category. Since every element in a non-degenerate $R$-module $M$ is in the image of a homomorphism $Rq\to M$ for some idempotent~$q$, the category of non-degenerate $R$-modules has enough projectives. Similarly there are enough injectives. So the derived functors of the hom and tensor product functors are available, and 
the usual notions and tools from homological algebra still work for idempotented rings. See~\cite[Chapter~XII]{borel+wallach}. The same discussion applies for 
right $\cH(G)$-modules. 
The Hecke algebra $\cH(G)$ is idempotented; it is the union of $1_K\cH(G)1_K=\cH(G,K)$, $K\in \mathcal{K}(G)$.

For a discrete subgroup $\Gamma<G$, we may now consider $\cH(G)$ as a $\cH(G)$-$\C[\Gamma]$-bimodule and use it to formulate a suitable Shapiro lemma for homology. Moreover, the induction by $\cH(G)$ can be explicitly computed in special cases, for example $\cH(G) \otimes_{\C[\Gamma]} \C = \cH(G/\Gamma)$, where $\cH(G/\Gamma)$ denotes the vector space of locally constant and compactly supported functions on the homogenous space $G/\Gamma$. 

\begin{lemma}[Shapiro]\label{lem: shapiro} Let $\Gamma$ be a discrete subgroup in $G$ and let $M$ be a left $\Gamma$-module. There is a natural isomorphism
$$H_*(\Gamma,M) \stackrel{\sim}{\to} {\rm Tor}_*^{\cH(G)}(\C,\cH(G) \otimes_{\C[\Gamma]} M).$$
\end{lemma}

\begin{proof}
	It suffices to prove that $\cH(G)$ is a flat right $\C[\Gamma]$-module. For a compact open subgroup $K<G$ the 
	right $\C[\Gamma]$-module $1_K\cH(G) $ is isomorphic to $\C[K\backslash G ]$ where $K\backslash G $ is a right 
	$\Gamma$-set with finite stabilisers. Such a $\C[\Gamma]$-module is projective. 
	Since $\cH(G)$ is the directed colimit 
	of right $\C[\Gamma]$-modules $1_K\cH(G)$ over $K\in\mathcal{K}(G)$, it is flat. 	
\end{proof}

\section{Spectral approximation in Hecke algebras}

Throughout, let $K<G$ be a compact open subgroup. We normalise the Haar measure $\nu$ on $G$ so that $\nu(K)=1$. 
For a lattice $\Gamma<G$ we denote by $\nu_{G/\Gamma}$ the finite measure on $X:=G/\Gamma$ induced by $\nu$. 
We do not normalise $\nu_{G/\Gamma}$; its total mass is $\covol_\nu(\Gamma)$, the 
$\nu$-covolume of~$\Gamma$. 

\subsection{Positive functionals on the Hecke algebra}\label{sub: positive functionals}

We denote by 
\[\bar{X} := K\backslash X = K \bs G / \Gamma\]
the countable set of double cosets and 
note that $\bar X$ carries a finite measure which gives the double coset $Ks\Gamma \in K \bs G / \Gamma$ the measure of its equivalence class $Ks\Gamma$, seen as a subset of $X$. 
Note that the canonical map $K \to Ks\Gamma$, $k \mapsto ks\Gamma$ is a finite covering with fibre of size equal to $|K \cap s\Gamma s^{-1}|$. Hence, we obtain $\nu_{G/\Gamma}(Ks\Gamma) = |K \cap s\Gamma s^{-1}|^{-1}$.
Thus, since $\Gamma$ is a lattice, we conclude
\begin{equation}\label{eq: covolume}
	\sum_{Ks\Gamma} \frac{1}{|K \cap s\Gamma s^{-1}|} = \sum_{Ks\Gamma\in \bar X}\nu_{G/\Gamma}(Ks\Gamma)=\nu_{G/\Gamma}(X)= \operatorname{covol}_{\nu}(\Gamma) < \infty.
\end{equation}
Further, let $\bar{X}_{e}$ be the set of double cosets $Ks\Gamma \in \bar{X}$ such that the map $K\owns k\mapsto ks\Gamma\in X$ is injective or, equivalently,
 \[K \cap s\Gamma s^{-1} = \{e\}.\]
 In particular, for each point $Ks\Gamma\in \bar{X}_{e}$, the corresponding equivalence class $Ks\Gamma \subseteq X$ has measure $\nu_{G/\Gamma}(Ks\Gamma) = 1$. It follows that the set $\bar{X}_{e}$ is finite, with cardinality 
\begin{equation}\label{eq: cardinality}
|\bar{X}_{e}| \leq \operatorname{covol}_{\nu}(\Gamma).
\end{equation}
More precisely, we obtain
\[|\bar{X}_{e}| = \nu_{G/\Gamma}(\{ s\Gamma \mid K \cap s\Gamma s^{-1} = \{e\} \}).\]
Thus, we see that $\ell^2(K\bs G/\Gamma,\nu)$ contains an isometric copy $\ell^2(\bar X_e)$. 
\begin{definition}
	We denote the projection of $\ell^2(K\bs G/\Gamma,\nu)$ onto the subspace $\ell^2(\bar X_e)$ by $$P_{\Gamma}\colon \ell^2(K\bs G/\Gamma,\nu)\to\ell^2(\bar X_e).$$ 
\end{definition}We want to understand in what sense the quotients of the form $K \bs G/\Gamma$ approximate $K \bs G$ as $\Gamma$ varies. Our emphasis is on the non-uniform case where special care is needed.
Indeed, we would like to consider a functional 
\[\tr_{\Gamma} \colon \cH(G,K) \to \C,\]
of the form
\[\tr_{\Gamma}(f) = \frac1{\operatorname{covol}_{\nu}(\Gamma)}\int_{G/\Gamma} \sum_{\gamma \in \Gamma} f(s \gamma s^{-1})\ d\nu_{G/\Gamma}(s),\]
which means more concretely
\[{\rm tr}_{\Gamma}(1_{KgK}) := \frac1{\operatorname{covol}_{\nu}(\Gamma)}\int_{G/\Gamma} |KgK \cap s\Gamma s^{-1}|\ d \nu_{G/\Gamma}(s).\]
If $\Gamma$ is a uniform lattice, then ${\rm tr}_{\Gamma}$ defines a (usually unnormalised) positive trace on $\cH(G,K)$, that can be used to relate the spectral properties of the action of the Hecke algebra on $\ell^2(K\bs G/\Gamma,\nu)$ to those of the action on $\ell^2(K \bs G)$.
However, since $\nu_{G/\Gamma}(Ks\Gamma) = |K \cap s\Gamma s^{-1}|^{-1}$, the integral for ${\rm tr}_{\Gamma}(1_K)$ will be infinite if the lattice is non-uniform. Thus, this approach is of no use in the general case.

In order to overcome this problem, we have to perform a renormalization procedure that concentrates on double cosets of full measure and control its defect on small double cosets. Note that the subspace of $L^2(G/\Gamma,\nu_{G/\Gamma})$ formed of $K$-invariant functions, which is spanned by a set of orthogonal functions $\{1_{Ks\Gamma} \mid Ks \Gamma \in \bar X \}$ is just $\ell^2(K\bs G/\Gamma,\nu)$, i.e. the weighted $\ell^2$-space on the set $K \bs G / \Gamma$ with weights as given above.
The Hilbert space $\ell^2(K\bs G/\Gamma,\nu)$ is endowed with an action 
\[\pi_\Gamma \colon \cH(G,K) \to B(\ell^2(K\bs G/\Gamma,\nu))\] 
given by convolution. We have 
\begin{equation*}
\langle \pi_\Gamma(1_{KgK})(1_{Kh\Gamma}), 1_{Ks\Gamma} \rangle=  \langle 1_{KgK} \ast 1_{Kh\Gamma}, 1_{Ks\Gamma}\rangle = \nu(KgK\cap s\Gamma h^{-1}K)\in\Z, 
\end{equation*}
where integrality follows from $KgK\cap s\Gamma h^{-1}K$ being 
a finite union of right $K$-cosets. We record for later: 

\begin{remark}\label{rem: integrality}
	Hence the action of $1_{KgK}$ on $\ell^2(K\bs G/\Gamma,\nu)$ 
	is given, with respect to the canonical basis consisting of indicator functions, by an infinite matrix  with integer entries and only finitely many non-zero entries in each row and column.
\end{remark}
	\begin{definition}\label{def: positive functional}
	The map 
\[ \varphi_\Gamma\colon\cH(G,K)\to\C,~~\varphi_{\Gamma} (T):= \frac{\sum_{Ks\Gamma\in \bar X} \langle \pi_\Gamma(T)(1_{Ks\Gamma}), 1_{Ks\Gamma} \rangle}{\operatorname{covol}_{\nu}(\Gamma)}.\]
defines a unital and positive linear functional on $\cH(G,K)$. The map 
\[ \varphi^e_\Gamma (1_{KgK})=\frac{\sum_{Ks\Gamma\in \bar X_e} \langle \pi_\Gamma(T)(1_{Ks\Gamma}), 1_{Ks\Gamma} \rangle}{\operatorname{covol}_{\nu}(\Gamma)}=\frac{\sum_{Ks\Gamma\in \bar X_e} \langle P_\Gamma\pi_\Gamma(T)P_\Gamma^\ast(1_{Ks\Gamma}), 1_{Ks\Gamma} \rangle}{\operatorname{covol}_{\nu}(\Gamma)}.\]
defines a (possibly non-unital) positive linear functional on $\cH(G,K)$. 
\end{definition} 

One easily verifies that 
\begin{equation}\label{rem: functional concrete}
\varphi_{\Gamma}(1_{KgK}) = \frac{\sum_{Ks\Gamma\in \bar X} \nu(KgK\cap s\Gamma s^{-1}K) \nu_{G/\Gamma}(Ks\Gamma)}{\operatorname{covol}_{\nu}(\Gamma)} \in [0,1].
\end{equation}
A similar identity holds for $\varphi_{\Gamma}^e(1_{KgK})$ with the sum only running over $Ks\Gamma\in \bar X_e$. 
This time, we obtain a positive functional $\varphi_\Gamma$ resembling the spectral properties of the action on $\ell^2(K\bs G/\Gamma,\nu)$,  no matter if $\Gamma$ is uniform or not.

\vspace{0.1cm}

The functionals $\tr$ and $\phi_{\Gamma_i}$ and $\phi_\Gamma^e$ (by first taking the matrix trace and then the functional) as well as other notions discussed so far extend to matrix algebras $M_n(\cH(G,K))$ for $n \in \N$ -- and we will use the notation introduced so far unchanged in the setting of matrix algebras.

\subsection{Spectral approximation for Farber sequences}
We are now finished with our preparations and will proceed by stating and proving the first main theorem.
Let $(\Gamma_i)_{i \in \N}$ be a sequence of lattices in $G$. Let $T\in M_n(\cH(G,K))$ be self-adjoint.  
By the Riesz representation theorem there is a unique Borel measure~$\mu_T$, called the \emph{spectral measure} with respect to $\tr$,  on $\R$ such that 
\[\tr(T^k) = \int_{\R} t^k\ d\mu_T(t).\]
Similarly, one defines the spectral measures $\mu_{T, i}$ and $\mu_{T,i}^e$ with respect to $\phi_{\Gamma_i}$ and 
$\phi_{\Gamma_i}^e$. The measure $\mu_T$ is supported in the interval 
$[-\lVert \pi(T)\rVert, \lVert \pi(T)\rVert]$. The measures $\mu_{T,i}$ and $\mu_{T,i}^e$ 
are supported in the interval $[-\lVert \pi_{\Gamma_i}(T)\rVert, \lVert \pi_{\Gamma_i}(T)\rVert]$.
Since the operator norms $\pi(T)$ and $\pi_{\Gamma_i}(T)$ are bounded 
by the maximum of the (obvious) $\ell^1$-norms of the entries of $T$ 
times $n^2$ (cf.~\cite{lueck-book}*{Lemma~13.33 on p.~470}), 
all three measures are supported on the compact interval $[-c, c]$ 
with 
\begin{equation}\label{eq: norm bound}
	c:=\max\{\lVert T_{i,j}\rVert_1\mid i,j\in\{1,\ldots, n\}\}\cdot n^2.
\end{equation}
Moreover, $\mu_T$ and $\mu_{T,i}$ are probability measures while the total mass 
of $\mu_{T,i}^e$ may be less than~$1$. 

\begin{theorem} \label{thm:spectral_appr}
Let $(\Gamma_i)_{i \in \N}$ be a Farber sequence of lattices in $G$. Let 
$T\in M_n(\Z[G,K])$ be a self-adjoint element. 
\begin{enumerate}
\item[$(i)$]
The sequences of measures $(\mu_{T,i})$ and $(\mu_{T,i}^e)$ both weakly converge 
to $\mu_T$. 
\item[$(ii)$] We have  
$\mu_{T}(\{0\})= \lim_{i\to\infty} \mu_{T,i}(\{0\})=\lim_{i\to\infty} \mu_{T,i}^e(\{0\}).$
\item[$(iii)$] Assume in addition that $T=S^\ast S$ is positive. 
Then we have 
\begin{equation} %\label{eq2}
\mu_{T}(\{0\})=\lim_{i\to\infty}\frac{\dim_\C \ker(P_{\Gamma_i}\pi_{\Gamma_i}(T)P_{\Gamma_i}^\ast)}{\covol_\nu(\Gamma_i)} \le  \liminf_{i \to \infty} \frac{\dim_{\C} \ker(\pi_{\Gamma_i}(T))}{\operatorname{covol}_{\nu}(\Gamma_i)}\nonumber	
\end{equation}
with equality provided that the sequence $(\Gamma_i)_{i \in \mathbb N}$ is uniformly discrete and $K$ sufficiently small. 
\end{enumerate}
\end{theorem}
\begin{proof}
\noindent (i) Every $\phi_{\Gamma_i}$ is unital, so $\phi_{\Gamma_i}(1_K)=1$. 
For $g\in G$ we have 
\[ \nu (KgK\cap s\Gamma_i s^{-1}K)=\nu ((KgK\cap s\Gamma_i s^{-1})\cdot K) \le  \begin{cases} 
 				                               \nu(KgK) & \text{ if $KgK\cap s\Gamma_i s^{-1}\ne\emptyset $, }\\
 				                               0 & \text{ otherwise.}
 \end{cases}
\]
Hence, using Equation \ref{rem: functional concrete}, 
\begin{align*} 
\phi_{\Gamma_i}(1_{KgK})&\le\nu(KgK)\cdot \sum_{\substack{Ks\Gamma_i\in K\backslash G/\Gamma_i\\ KgK\cap s\Gamma_i s^{-1}\ne\emptyset}} \frac{\nu_{G/\Gamma_i}(Ks\Gamma_i)}{\operatorname{covol}_\nu(\Gamma_i)}\\
&=
    \nu(KgK)\cdot\frac{\nu_{G/\Gamma_i}\bigl(\{s\Gamma_i\mid KgK\cap s\Gamma_i s^{-1}\ne\emptyset\}\bigr)}{\operatorname{covol}_\nu(\Gamma_i)}\\
    &= \nu(KgK)\cdot\mu_{\Gamma_i}\bigl(\{H\in\sub_G\mid H\cap KgK\ne \emptyset\}\bigr)
\end{align*}
For $g\not\in K$ the latter tends to zero as $i\to\infty$ by the Farber condition as $KgK$ can be covered by finitely many non-trivial $K$-cosets. Since $\phi_{\Gamma_i}$ is unital we obtain that 
\begin{equation}\label{eq: elementwise convergence}
	\lim_{i\to\infty} \phi_{\Gamma_i}(1_{KgK}) = \tr(1_{KgK})\nonumber
\end{equation}
for every double coset $KgK$. Since $0\le \phi_{\Gamma_i}^e(1_{Kg\Gamma})\le \phi_{\Gamma_i}(1_{Kg\Gamma})$ we obtain also that 
\[\lim_{i\to\infty} \phi_{\Gamma_i}^e(1_{KgK}) = \tr(1_{KgK})=0~\text{ for $g\not\in K$.}\]
And we have 
\begin{align*} 
	\phi_{\Gamma_i}^e(1_K)~=~\sum_{{\substack{Ks\Gamma_i\\ K\cap s\Gamma_i s^{-1}=\{e\}}}}~\frac{\nu(Ks\Gamma_i)}{\covol_\nu(\Gamma_i)}&=\frac{\bigl|\{Ks\Gamma\in K\backslash G/\Gamma\mid K\cap s\Gamma_i s^{-1}=\{e\}\}\bigr |}{\covol_\nu(\Gamma_i)}\\
	&=\mu_{\Gamma_i}\bigl(\{H\in \sub_G\mid H\cap s\Gamma_i s^{-1}=\{e\}\}\bigr).
\end{align*}
By the Farber condition $\phi_{\Gamma_i}^e(1_K)$ tends to $\tr(1_K)=1$ as $i\to\infty$. 
This implies that the spectral measures $\mu_{T,i}$ and $\mu_{T,i}^e$ both converge to $\mu_T$ in moments. 
Since all the measures are supported on the same compact interval, $\mu_{T,i}$ and 
$\mu_{T,i}^e$ weakly 
converge to $\mu_T$. 

\noindent (ii) By the Portmanteau theorem weak convergence is equivalent  to 
\begin{equation}
\mu_T(E) \geq \limsup_i \mu_{T,i}(F), \quad \mu_T(U) \leq \liminf_i \mu_{T,i}(U). \nonumber
\end{equation}
for any closed subset $F\subset \R$ and any open subset $U\subset\R$. Similarly 
for $\mu_{T,i}^e$. In particular, we have 
$\limsup \mu_{T,i}(\{0\})\le \mu_T(\{0\})$ 
and $\limsup \mu_{T,i}^e(\{0\})\le \mu_T(\{0\})$. 
Clearly, $\liminf \mu_{T,i}^e(\{0\})\le \liminf \mu_{T,i}(\{0\})$. So it remains to 
show that 
\begin{equation}\label{eq: liminf ineq}
\liminf_{i\to\infty}\mu_{T,i}^e(\{0\}) \ge \mu_T(\{0\}). 	
\end{equation}
The basic mechanism for proving this builds on the integrality of $T$ and goes back to the work of L\"uck~\cite{Lu94}. We apply this mechanism in our setting. 

Let $c>0$ be as in~\eqref{eq: norm bound}. 
Fix $i\in\N$. It is clear that $P_{\Gamma_i}\pi_{\Gamma_i}(T)P_{\Gamma_i}^\ast$ is a finite-dimensional self-adjoint operator. 
Let $e_1, \dots ,e_{l}$ be its eigenvalues written with multiplicity and ordered by increasing absolute value, and let $e_{m+1}$ be the first non-zero eigenvalue. Since $P_{\Gamma_i}\pi_{\Gamma_i}(T)P_{\Gamma_i}^\ast$ acts as an $l\times l$ matrix with integer entries on $\ell^2(\bar X_e)$, we have $\lvert e_{m+1}\cdots e_l\rvert \geq 1$ and $\lvert e_l\rvert \leq c$. By~\eqref{eq: cardinality} we have $l\leq |\bar X_e|\leq\operatorname{covol}_{\nu}(\Gamma_i)$. 
Fix an $\varepsilon > 0$ and let $\delta:= |\{ i \mid e_i \in (-\varepsilon ,\varepsilon ) \setminus \{0\} \}|$. Then $1\leq \varepsilon^{\delta}\cdot c^{l}$. Note that 
$\mu_{T,i}^e( (-\varepsilon ,\varepsilon)\setminus \{0\})$ is the matrix trace, normalised by $\covol_\nu(\Gamma_i)$,  
of the projection onto the sum of eigenspaces of $P_{\Gamma_i}\pi_{\Gamma_i}(T)P_{\Gamma_i}^\ast$  corresponding to eigenvalues in $(-\epsilon, \epsilon)\backslash\{0\}$ 
. Hence $\mu_{T,i}^e( (-\varepsilon ,\varepsilon)\setminus \{0\})=\delta/\covol_\nu(\Gamma_i)$. 
It follows that
\begin{equation*}
\mu_{T,i}^e( (-\varepsilon ,\varepsilon)\setminus \{0\}) = 
\frac{\delta}{\covol_\nu(\Gamma_i)}\le \frac{\delta}{l} \le \frac{\log c}{\lvert \log \varepsilon \rvert}. 
\end{equation*}
Since $i\in\N$ was arbitrary, we conclude that, for any $\varepsilon > 0$,
\begin{align*}
\liminf_{i\to\infty} \mu_{T,i}^e(\{0\}) &\geq \liminf_{i \to \infty} \mu_{T,i}^e((-\varepsilon,\varepsilon)) - \frac{\log c}{\lvert \log \varepsilon \rvert} \\
&\geq \mu_T((-\varepsilon, \varepsilon)) - \frac{\log c}{\lvert \log \varepsilon \rvert} \\
&\geq \mu_T(\{0\}) - \frac{\log c}{\lvert \log \varepsilon \rvert}. \nonumber
\end{align*}
Since $\varepsilon>0$ was arbitrary we conclude~\eqref{eq: liminf ineq}. 

\noindent (iii) By definition, $\mu_{T,i}^e(\{0\})$ is the 
matrix trace, normalised by $\covol_\nu(\Gamma_i)$, of the projection of $\ell^2(\bar X_e)$ onto the 
kernel of $P_{\Gamma_i}\pi_{\Gamma_i}(T)P_{\Gamma_i}^\ast$. Hence $\mu_{T,i}^e(\{0\})$ 
is just the vector space dimension of $\ker(P_{\Gamma_i}\pi_{\Gamma_i}(T)P_{\Gamma_i}^\ast)$ normalised by $\covol_\nu(\Gamma_i)$. By positivity of 
$\pi_{\Gamma_i}(T)$ we have 
\[\dim_\C \ker(P_{\Gamma_i}\pi_{\Gamma_i}(T)P_{\Gamma_i}^\ast)
\le \dim_\C \ker(\pi_{\Gamma_i}(T)),\]
thus the stated inequality follows. If $(\Gamma_i)_{i \in \mathbb N}$ is uniformly discrete and if we take $K$ sufficiently small, then 
$P_{\Gamma_i}=\id$ for $i$ sufficiently large, and we get equality. This finishes the proof.
\end{proof}

\section{$L^2$-Betti numbers of totally disconnected groups}

In this section we review the definition of $L^2$-Betti numbers of totally disconnected groups from~\cite{Pe11} and provide some additional tools. 
The algebraic approach pioneered by L\"uck \cite{lueck-book} has found various applications in ergodic theory, algebra and geometry, see \cites{MR2827095, MR2399103, MR2486803, MR2572246, MR2650795}.
We conclude the proof of Theorem~\ref{thm:main} in Subsection~\ref{sub: proof of main}.

\subsection{The group von Neumann algebra}

We denote by $$\lambda \colon \cH(G) \to B(L^2(G,\nu))$$ the left-regular representation and by $\rho \colon \cH(G) \to B(L^2(G,\nu))$ the right-regular representation -- defined by left and right convolution, respectively. The \emph{group von Neumann algebra} of $G$ is defined to be the weak closure \[L(G) = \overline{\lambda(\cH(G))}^w.\] 
The corresponding closure of the image of $\rho$ is naturally anti-isomorphic to $L(G)$, so that $L^2(G,\nu)$ becomes a $L(G)$-bimodule in a natural way.

Let $\cH(G)_+$ and $L(G)_+$ denote the subsets of positive elements $x^\ast x$, respectively. 
The trace ${\rm tr}$ in Definition~\ref{def: trace on hecke algebra} is independent of the Haar measure, but the algebra structure of $\cH(G)$ depends on the choice of a Haar measure, being defined in terms of convolution. 
The (restricted) trace ${\rm tr}\vert_{\cH(G)_+}\colon \cH(G)_+\to [0,\infty)$ 
extends to a faithful normal semifinite trace ${\rm tr}\colon L(G)_+\to [0,\infty]$. Further, it extends to a faithful normal semifinite trace on the von Neumann algebra of $n\times n$-matrices over $L(G)$ which we denote by the same symbol. In particular, $L(G)$ and $M_n(L(G))$ are semifinite von Neumann algebras. 
Since the linear span of positive elements with finite trace is dense in $L(G)$, 
${\rm tr}$ induces a densely defined, faithful, positive tracial weight on $L(G)$. 

Furthermore, if $p\in L(G)$ is a projection with $\tr(p)<\infty$, 
	then $pL(G)p$ is a finite von Neumann algebra with the restriction of $\tr$ as finite trace. We usually normalise this trace with $1/\tr(p)$. 

We refer to~\cite{sauer-survey} for a more detailed survey of the above notions. 

\subsection{The dimension for modules over the group von Neumann algebra}
	We review the definition of dimension of arbitrary right $L(G)$-modules 
	in~\cite{Pe11} which generalizes the corresponding work of L\"uck for discrete groups. 
	
	The dimension of a finitely generated projective (right) $L(G)$-module $P=pL(G)^n$ where 
	$p$ is a projection in $M_n(L(G))$ is defined as 
	\[\dim_{(L(G),{\rm tr})} (P) := \tr(p) \in [0,\infty].\]
    For an arbitrary $L(G)$-module $M$ one defines 
    \[\dim_{(L(G),{\rm tr})} (M) := \sup\{ \dim_{(L(G),{\rm tr})}(P)\mid P\subset M~\text{ f.g.~proj.~submodule}\} \in [0,\infty].\]
    This dimension is additive for short exact sequences and continuous with respect to ascending unions of modules~\cite{Pe11}*{Theorems~B.22 and~B.23}. 
    \begin{remark}
      If $q$ is a projection in $M_n(L(G))$ and $M=qL^2(G,\nu)^n$ is 
    the image of $q$ then $\dim_{(L(G),{\rm tr})} (M)={\rm tr}(q)$ 
    by~\cite{Pe11}*{Theorem~B.25}. 
    If $K \subset G$ is a compact open subgroup and $p_K$ denotes the projection onto the left $K$-invariant vectors in $L^2(G,\nu)$, then we obtain
\[\dim_{(L(G),{\rm tr})} (\ell^2(K \bs G)) = \dim_{(L(G),{\rm tr})} (p_KL^2(G,\nu)) = {\rm tr}(p_K) = \frac{1}{\nu(K)}.\]
\end{remark}

Next we introduce some tools that are needed to express $L^2$-Betti numbers 
of totally disconnected groups via the dimension over a finite von Neumann algebra (see Lemma~\ref{lem: comparision of l2 betti}).

\begin{remark}\label{rem: inclusion of compact open subgroups}
If $K<K'$ is an inclusion of compact open subgroup in $G$, then the projections satisfy $p_{K'}\le p_{K}$.
For two projections $p, q$ in a von Neumann algebra we write $p\sim q$ if $p,q$ 
are Murray-von Neumann equivalent, and we write $p\preceq q$ if there 
is $\tilde p\sim p$ such that $\tilde p\le q$. 
If the subgroup $K$ is subconjugated to $K'$, then 
$p_{K'}\preceq p_K$. 
\end{remark}

\begin{definition}
	Let $M$ be a $L(G)$-module. Let $p\in L(G)$ be a projection. The \emph{support} $s(x)\in L$ of an element $x\in M$ is the smallest projection $s(x)\in L$ such that $xs(x)=x$. 
	We say that $M$ is \emph{$p$-truncated} if $s(x)\preceq p$ 
	holds for every $x\in M$. 
\end{definition}

\begin{remark}\label{rem: truncated modules}
The class of $p$-truncated modules over $L(G)$ is closed 
under taking submodules and homomorphic images. The prototypical case 
of a $p$-truncated $L(G)$-module is $pL(G)$: Let $x\in pL(G)$. Let $s_l(x)$ be 
the smallest projection in $L(G)$ with $s_l(x)x=x$. Clearly, we have 
$s_l(x)\le p$. By polar decomposition, $s(x)\sim s_l(x)$ 
(cf.~\cite[Proposition~1.5 on p.~292]{takesaki}). 
\end{remark}

The following lemma is crucial for relating various definitions of $L^2$-Betti numbers that exist in the literature.

\begin{lemma}\label{lem: dimension and cutting down}
Let $p\in L(G)$ be a projection with $\tr(p)<\infty$. 
Let $M$ be a $p$-truncated $L(G)$-module. We have 
\begin{equation*}
		\dim_{(L(G),\tr)} (M)=\tr(p)\cdot \dim_{(pL(G)p, \tr(p)^{-1}\tr)}(pM).
\end{equation*}
\end{lemma}

\begin{proof}
First we show the assertion for finitely generated projective modules of finite 
$L(G)$-dimension.  
Let $q\in M_n(L(G))$ be a projection such that the $L(G)$-module $qL(G)^n$ is 
$p$-truncated. Let $e_i\in L(G)^n$ be 
the $i$-th standard basis vector. Let $x_i=qe_i\in P$. By assumption 
$s(x_i)\preceq p$. Since $s(x_i)$ and $p$ are finite projections there 
are unitaries $u_i\in L$ such that $s(x_i)\le u_i^\ast pu_i$~\cite[Proposition~1.38 on p.~304]{takesaki}. 
Hence 
\[ q\cdot \diag(u_1^\ast p u_1, \ldots, u_n^\ast p u_n)=q.\]
This implies 
\[q\preceq\diag(p,\ldots, p)\sim \diag(u_1^\ast p u_1, \ldots, u_n^\ast p u_n).\]
Let $q'\sim q$ a projection in $M_n(L(G))$ such that $q'\le \diag(p,\ldots, p)$. 
Then 
\begin{align*}
	\dim_{(pL(G)p,\tr(p)^{-1}\tr)}(qL(G)^np) &= \dim_{(pL(G)p,\tr(p)^{-1}\tr)}(q'L(G)^np)\\
	&=\dim_{(pL(G)p,\tr(p)^{-1}\tr)}(q'pL(G)^np)\\
	&=\tr(q')\tr(p)^{-1}\\
	&=\tr(q)\tr(p)^{-1}\\
	&=\dim_{(L(G),\tr)} (qL(G)^n)\tr(p)^{-1}.
\end{align*}
Let $M$ be now an arbitrary $L(G)$-module. 
For every finitely generated projective submodule $P$ of $M$,  
$Pp$ is a submodule of $Mp$ whose dimension with respect 
to $pL(G)p$ is $\dim_{(L(G),\tr)}(P)\tr(p)^{-1}$. 
This implies the 
$\le$-inequality in the statement of the lemma. For the 
$\ge$-inequality we refer to the proof of~\cite[B.35 Theorem]{Pe11}.  
\end{proof}

\subsection{{$L^2$}-Betti numbers}\label{sub: l2 betti numbers}

The $L^2$-Betti numbers of a locally compact, separable, and unimodular group are defined as
\[ \beta^{(2)}_n(G,\nu) := \dim_{(L(G),{\rm tr})} H^n_c(G, L^2(G)),\]
where $H^n_c$ denotes the continuous group cohomology~\cite{Pe11}. 
In the case where $G$ is totally disconnected, it is shown in~\cite{Pe11} that 
\begin{equation}
\beta_n^{(2)}(G,\nu) := \dim_{(L(G),{\rm tr})} \tor^{\cH(G)}_n(\C,L_{\infty}^2(G,\nu)). \nonumber
\end{equation}
Here 
\[L^2_{\infty}(G) := \bigcup_{K \in \mathcal{K}(G)} {\ell^2(K \backslash G,\nu)}\subset L^2 (G,\nu)\]
is the vector space of smooth vectors which is naturally a $\cH(G)$-$L(G)$-bimodule. 

From a topological model $X$ of $G$ we obtain projective resolutions of $\C$ in the category 
of (non-degenerate) $\cH(G)$-modules which compute the above Tor group: Upon taking inverses we may assume that $G$ acts from the right on $X$. 
The $n$-skeleton 
$X^{(n)}$ is built from $X^{(n-1)}$ by attaching $G$-orbits of 
$n$-cells according to pushouts of $G$-spaces of the form:  
\begin{equation}\label{eq: pushout}
\xymatrix{
\bigsqcup_{U\in \cF_n}U\backslash G\times S^{n-1}\ar[d]\ar[r] & X^{(n-1)}\ar[d]\\ 
\bigsqcup_{U\in \cF_n}U\backslash G\times D^n\ar[r]     & X^{(n)}
}
\end{equation}

%\begin{equation}
%	\begin{tikzcd}
%	\bigsqcup_{U\in \cF_n}U\backslash G\times S^{n-1}\arrow[hook,d]\ar[r] & X^{(n-1)}\arrow[hook,d]\\
%	\bigsqcup_{U\in \cF_n}U\backslash G\times D^n\arrow[r]     & X^{(n)}
%	\end{tikzcd}
%\end{equation}
Here $\cF_n$ is the set of $n$-cells, which we loosely identify with a multi-set of representatives of 
conjugacy classes 
of stabilisers of $n$-cells. Each coset space $U\backslash G$ is discrete. 
Let us fix a choice of pushouts, which is not part of the data of a cellular complex and 
corresponds to an equivariant choice of orientations for the cells. 
The horizontal maps induce an isomorphism in relative homology by excision. 
Let $C_\ast(X)$ be the cellular chain complex. Every chain group $C_n(X)$ is 
a \emph{discrete $G$-module}, i.e.~an abelian group with a right $G$-action by 
homomorphisms such that the stabiliser of any element is open. The differentials 
are $G$-equivariant. 
We obtain isomorphisms of discrete $G$-modules: 
\begin{equation}\label{eq: geometric chain group}
	\bigoplus_{U\in \cF_n}\Z[G/U]\cong H_n(\bigsqcup_{U\in \cF_n} G/U\times (D^n, S^{n-1}))\xrightarrow{\cong} H_n(X^{(n)}, X^{(n-1)})\overset{\mathrm{def}}{=}C_n(X)
\end{equation}
\begin{remark}\label{rem: discrete G modules}
For any discrete $G$-module~$M$, $\C\otimes_\Z M$ is naturally a right $\cH(G)$-module via 
\[ mf=\int_G f(g)mg d\nu(g)~\text{ for $f\in \cH(G)$ and $m\in \C\otimes_\Z M$.}\]
This is has to be understood as follows: If $U<G$ is the stabiliser of $m$, then read the integral as the 
finite sum $\sum_{g\in G/U} (\int_U f(gu)d\nu(u))mg$. Any homomorphism of discrete $G$-modules becomes thus a homomorphism of $\cH(G)$-modules. 
\end{remark}
Hence the cellular chain complex $C_\ast(X;\C)$ with complex coefficients is a chain complex of $\cH(G)$-modules. Since $\C[U\backslash G]$ is isomorphic to 
$p_U\cH(G)$ when regarded as a $\cH(G)$-module, $C_\ast(X;\C)$ is a projective 
resolution of $\C\cong H_0(X)$ in the category of non-degenerate right $\cH(G)$-modules. Hence we can conclude: 

\begin{remark}\label{rem: top model computes l2}
Let $X$ be a topological model of $G$. Then 
\begin{equation}
\beta_n^{(2)}(G,\nu) = \dim_{(L(G),\tr)} H_n\bigl( C_\ast(X;\C)\otimes_{\cH(G)}L_{\infty}^2(G,\nu)\bigr). \nonumber
\end{equation}
\end{remark}

The appropriate notion of Euler characteristic in this context counts a $G$-cell of the form $U \backslash G \times D^n$ with weight $(-1)^n \cdot \nu(U)^{-1}$: 

\begin{definition}
	If $G$ has a cocompact topological model, then the \emph{Euler characteristic} of $G$ with respect to the Haar measure $\nu$ is -- in reference to the setting in~\eqref{eq: geometric chain group} -- defined as  
	    \[\chi(G,\nu)=\sum_{p\ge 0}(-1)^p\sum_{U\in \mathcal{F}_n} \nu(U)^{-1}.\]
\end{definition}

By the same argument as in~\cite{lueck-book}*{Theorem~1.35}, which is soley based on the additivity of the dimension, we obtain the Euler-Poincare formula: 

\begin{theorem}
If $G$ has a cocompact topological model, then 
\[\chi(G,\nu)=\sum_{p\ge 0}(-1)^p \betti_p(G,\nu).\]
\end{theorem}

\subsection{Conclusion of proof of Theorem~\ref{thm:main}}\label{sub: proof of main}
We retain the notation of Theorem~\ref{thm:main}. Let $X$ be a topological model 
of $G$ whose $(n+1)$-skeleton is cocompact, i.e.~has only finitely many $G$-orbits of cells up to dimension~$(n+1)$. We pick $G$-pushouts as 
in~\eqref{eq: pushout}. Let $K$ be the (finite) intersection of the corresponding stabiliser subgroups $U\in \cF_i, i\in\{0,\ldots, n+1\}$. Without loss of generality we normalise the Haar measure so that $\nu(K)=1$. 

For a compact open subgroup $L>K$ the discrete $G$-module $\Z[K\backslash G]$ 
is isomorphic to $[L:K]$ many copies of $\Z[L\backslash G]$. Each $C_k(X)$ with 
$k\in\{0,\ldots, n+1\}$ is 
a finite sum of discrete $G$-modules of the form $\Z[L\backslash G]$ for $K<L$. 
By a standard method from homological algebra we may add 
short exact chain complexes of the form $\Z[L\backslash G]\xrightarrow{\id}\Z[L\backslash G]$, $K<L$, to the resolution $C_\ast(X)$ starting in degree~$0$ up to degree $(n+1)$
so that we obtain a resolution of $\Z$ by discrete $G$-modules $F_\ast$ with 
differentials $\delta_\ast$ 
where each $F_k$, $k\in\{0,\ldots, n+1\}$, is a finite sum of copies of $\Z[K\backslash G]$. Let $(F^\C_\ast,\delta_\ast^\C)$ be the complexification 
of $(F_\ast, \delta_\ast)$. It follows from the 
discussion in the previous subsection that $F_\ast^\C$ yields a projective 
$\cH(G)$-resolution. 
Accordingly the $L^2$-Betti numbers of $G$ are computed as the $L(G)$-dimensions of the homology of $F^\C_\ast\otimes_{\cH(G)} L_{\infty}^2(G,\nu)$. For each degree $k\in\{0,\ldots, n+1\}$ let $a_k$ be the number 
of summands $\Z[K\backslash G]$ in $F_k$. Then we have the canonical isomorphisms of $L(G)$-modules: 
\begin{align*}
F_k^\C\otimes_{\cH(G)} L^2_\infty(G, \nu)\cong \C[K\backslash G]^{a_k}\otimes_{\cH(G)} L^2_\infty(G,\nu)&\cong(p_K\cH(G))^{a_k}\otimes_{\cH(G)} L^2_\infty(G,\nu)\\
&\cong (p_KL^2_\infty(G,\nu))^{a_k}\\
&\cong L^2(K\backslash G)^{a_k}.
\end{align*}
By Lemma~\ref{lem: identification G-equivariant maps} the differentials $\delta_\ast$ in degrees $\le n+1$ 
are given by convolution with finite-dimensional matrices over $\Z[G,K]$. Hence the same is true for 
the differentials $\delta_\ast^{(2)}=\delta_\ast^\C\otimes\id$ of $F_\ast^\C\otimes_{\cH(G)} L^2_\infty(G, \nu)$. 
The Hilbert adjoint of $\delta_k^{(2)}$ is induced from the formal adjoint of $\delta_k$ which also comes from a 
matrix with entries in $\Z[G,K]$. Hence the combinatorial Laplacian $\Delta_k=\delta_{k+1}^{(2)}(\delta_{k+1}^{(2)})^\ast+ (\delta_{k}^{(2)})^\ast\delta_{k}^{(2)}$ in degrees $k\in\{0,\ldots, n\}$ is given by an element $T_k\in M_{a_k}(\Z[G,K])$, or in the notation of Subsection~\ref{sub: hecke pair}, $\Delta_k=\pi(T_k)$.  
 We then have the analogue of the usual $L^2$-Hodge theorem, which yields 
\begin{equation}
\beta^{(2)}_k(G,\nu) = \dim_{(L(G),{\rm tr})} \ker\left(\Delta_k \colon L^2(K\backslash G)^{a_k} \to L^2(K\backslash G)^{a_k} \right). \nonumber
\end{equation}
In the sequel we refer a lot to the notation of Subsection~\ref{sub: positive functionals}. 
Let now $\Gamma_i<G$ be a lattice from the Farber sequence. 
By Lemma~\ref{lem: shapiro}, 
\[H_k(\Gamma_i,\mathbb{C}) \cong \operatorname{Tor}_k^{\mathcal{H}(G)}(\C, \cH(G)\otimes_{\C[\Gamma_i]}\C)=H_k(F_\ast^\C\otimes_{\C[\Gamma_i]}\C).\] 
Note that $F_k^\C\otimes_{\C[\Gamma_i]}\C$ is canonically isomorphic to $\C[K\backslash G/\Gamma_i]^{a_k}$. 
We 
can consider the combinatorial Laplacian on $F_k^\C\otimes_{\C[\Gamma_i]}\C=\C[K\backslash G/\Gamma_i]^{a_k}$ 
and on its completion $\ell^2(K\backslash G/\Gamma_i,\nu)^{a_k}$ (with the weighted measure on $K\bs G/\Gamma$). 
We denote by $\Delta_k^{\Gamma_i}$ the combinatorial Laplacian on the completion. 
The set $X_e$ of elements in $K\bs G/\Gamma$ with measure~$1$ is finite and we have the inclusions 
\[ \C[X_{e,i}]^{a_k}\subset \C[K\backslash G/\Gamma_i]^{a_k}\subset \ell^2(K\backslash G/\Gamma_i,\nu)^{a_k}\]
As in Subsection~\ref{sub: positive functionals} we write $\C[X_{e,i}]$ as $\ell^2 (X_e)$ to indicate that 
as a subspace of the Hilbert space $\ell^2(K\bs G/\Gamma,\nu)$ the set $X_{e,i}$ is a Hilbert basis of $\C[X_{e,i}]$. 
Let $T_k\in M_{a_k}(\Z[G,K])\subset M_{a_k}(\cH(G;K))$ be the matrix representing $\Delta_k$. 
We have
\[  \Delta_k^{\Gamma_i} = \pi_{\Gamma_i}(T_k).\]
Moreover, 
\[ \ker\bigl( P_{\Gamma_i}\Delta_k^{\Gamma_i}P_{\Gamma_i}^\ast\bigr)=\ker\bigl(\Delta_k^{\Gamma_i}\bigr)\cap \ell^2 (X_{e,i})^{a_k}=\ker \bigl( \Delta_k^{\Gamma_i}\vert_{\ell^2(X_{e,i})^{a_k}}\bigr).\]
We have 
\begin{equation} \label{eq:beta_compare}
\dim_{\mathbb{C}}\ker\bigl( P_{\Gamma_i}\Delta_k^{\Gamma_i}P_{\Gamma_i}^\ast\bigr)\leq \beta_n(\Gamma_i),
\end{equation}
since any non-zero element in $\ker\bigl(\Delta_k^{\Gamma_i}\bigr)\cap \ell^2 (X_{e,i})^{a_k}$ is obviously 
a cycle in $F_\ast^\C\otimes_{\C[\Gamma_i]}\C$ and at the same time orthogonal to the image of the next 
differential, thus gives rise to a non-zero class in homology. 
According to Theorem \ref{thm:spectral_appr} we have 
\begin{align*}
\lim_{i \to \infty}\frac{\dim_{\mathbb{C}}\ker\bigl( P_{\Gamma_i}\Delta_k^{\Gamma_i}P_{\Gamma_i}^\ast\bigr) }{\operatorname{covol}_{\nu}(\Gamma_i)}&= \mu_{T_k}(\{0\})\\
&= \dim_{(L(G), {\rm tr})} \ker \pi(T_k)\\
&= \dim_{(L(G), {\rm tr})} \ker \Delta_k\\
&= \beta_n^{(2)}(G,\nu). 
\end{align*}
By~\eqref{eq:beta_compare} this implies immediately that
\begin{equation}
\beta^{(2)}_n(G,\nu) \leq \liminf_{i \to \infty} \frac{\beta_n(\Gamma_i)}{\operatorname{covol}_{\nu}(\Gamma_i)}. \nonumber
\end{equation}
with equality in case of a uniformly discrete Farber sequence of lattices.
This proves Theorem~\ref{thm:main}. 

\section{Applications and examples of groups}\label{sec: computations}

\subsection{Automorphism groups of buildings}
$L^2$-Betti numbers of buildings were intensively studied by Davis, Dymara, Januszkiewicz and Okun~\cites{dymara, davisetal}. In this subsection we relate our treatment of $L^2$-Betti numbers to theirs. To this end, we prove the following lemma which is also of independent interest as it allows to compute $L^2$-Betti numbers via the dimension function over finite von 
Neumann algebras.

\begin{lemma}\label{lem: comparision of l2 betti}
Let $X$ be a topological model of~$G$. Let $B<G$ be a compact open subgroup such 
that for every $x\in X^{(n+1)}$ there is $g\in G$ with $gBg^{-1}\subset G_x$. 
We normalise the Haar measure $\nu$ so that $\nu(B)=1$. 
Then with respect to the normalised trace on the finite von Neumann algebra 
$p_B L(G)p_B$ we have 
\begin{align*}
\beta_i^{(2)}(G,\nu)&=\dim_{p_BL(G)p_B} H_i(G, \ell^2(G/B))\\
					&=\dim_{p_BL(G)p_B} \tor_i^{\Z[G,B]} (\Z, \ell^2(B\bs G/B,\nu))
\end{align*}
for $i\in\{0,\ldots, n\}$. 
\end{lemma}

\begin{proof}
The cellular chain complex $C_\ast=C_\ast(X;\C)$ with complex coefficients 
is a (right) projective $\cH(G)$-resolution of $\C$. 
Each chain group  
$C_i\otimes_{\cH(G)} L^2_\infty(G)$ for $i\in\{0,\ldots, n+1\}$ 
is by assumption a direct sum of (right) $L(G)$-modules of type 
$p_K L^2_\infty(G)$ where $K<G$ is a compact open subgroup such that $B$ is 
subconjugated to $K$. Hence by Remarks~\ref{rem: inclusion of compact open subgroups} and~\ref{rem: truncated modules} the chain complex and 
thus $H_\ast(G, L^2_\infty(G))$ are $p_B$-truncated up to degree~$n$. 
By Lemma~\ref{lem: dimension and cutting down} we have 
\[\beta_i^{(2)} (G,\nu)=\dim_{p_BL(G)p_B} \bigl(H_i(G, L^2_\infty(G))p_B\bigr)\]
in this range. Since $p_B$ is an idempotent we can pull $p_B$ inside by exactness 
and obtain the first equality of the lemma. 
For the second equality note that $C_ip_B$ is a projective $\Z[G;B]$-module 
for $i\in\{0,\ldots, n+1\}$. Since $C_i$ is a direct sum of modules of 
type $\Z[K\bs G]$ with $B$ being subconjugated to $K$ (see above) it suffices 
to prove that $\Z[K\bs G]p_B=\Z[K\bs G/B]$ is a projective $\Z[G;B]$-module: 
Let $g\in B$ be such that $B\subset gKg^{-1}$. Then 
\begin{align*} \Z[K\bs G]p_B\cong \Z[gKg^{-1}\bs G]p_B&\cong p_{gKg^{-1}}\cH(G)p_B\\
&=p_{gKg^{-1}}p_B\cH(G)p_B\\
&=p_{gKg^{-1}}\Z[G;B].
\end{align*}

Furthermore, one easily verifies that 
\[ \Z[K\bs G]p_B\otimes_{\Z[G;B]} \ell^2 (B\bs G/B,\nu)\cong \Z[K\bs G]\otimes_{\cH(G)} \ell^2(G/B).\]
Thus we have 
\[ C_ip_B\otimes_{\Z[G;B]} \ell^2 (B\bs G/B,\nu)\cong C_i\otimes_{\cH(G)}\ell^2 (G/B)\]
for $i\in\{0,\ldots, n+1\}$ 
from which we conclude the second equality in the statement. 
\end{proof}

For the remainder of this subsection we agree on the following setup. 

\begin{setup}\label{setup} \normalfont
Let $\Phi$ be a building of (Coxeter) type $(W,S)$ with a chamber transitive automorphism group~$G$. Let $X$ be the geometric realisation of $\Phi$. We endow $G$ with the subspace topology of the homeomorphism group of $X$ and assume that $G$ is unimodular. We fix a chamber and denote its stabiliser by $B$; it is a compact-open subgroup of $G$. We normalise the Haar measure on $G$ to satisfy $\nu(B)=1$. 
\end{setup}

%
%\begin{lemma}
%	Let $p_B\in L(G)$ be the projection associated 
%	to compact open subgroup $B<G$. Then $H_i(G, L^2(G))$ is $p_B$-truncated 
%	for every $i\ge 0$. 
%\end{lemma}

The geometric realisation $X$ (in the sense 
of~\cite{davis-book}) is a CAT(0)-space on which $G$ acts with compact-open 
stabiliser. 
Hence $X$ is a topological model of $G$. The group $B$ is the stabiliser of the fundamental chamber, and $X$ satisfies the assumptions of the previous theorem. Furthermore, the complex-valued cellular chain complex of the \emph{Davis complex} $\Sigma=X/B$ is a projective resolution over 
$\cH(G;B)$.

As a consequence of Lemma~\ref{lem: comparision of l2 betti} we deduce the following statement which 
was known before~\cite{dymara}*{Theorem~3.5}. 

\begin{lemma}
	The $L^2$-Betti numbers of $G$ coincide 
	with the $L^2$-Betti numbers $L^2_qb^i(\Sigma)$ as defined in~\cite{dymara}. 
\end{lemma}

\subsection{Chevalley groups and their lattices in finite characteristic}\label{sub: chevalley}
We retain Setup~\ref{setup}. 
There is the following formula relating the \emph{Poincar\'{e} series} 
\[ \omega_{(W,S)}(t)=\sum_{w\in W} t^{l(w)}\in \Z[[t]] \]
of the Coxeter system $(W,S)$, i.e. the type of the building, to the Euler characteristic of $G$. We drop the subscript in $\omega_{(W,S)}$ if 
the Coxeter system is obvious from the context. As we review below, for an affine Coxeter system the 
formal series $\omega(t)$ is a rational function whose poles are roots of unity. 
See~\cite{davisetal}*{Section~3} for a nice discussion of~$\omega(t)$. 

\begin{theorem}[\cite{dymara}*{Corollary~3.4}]\label{thm: euler char} Let $X$ be of (uniform) thickness~$q$. With the normalisation of $\nu$ from Setup~\ref{setup} we have 
\[\chi(G,\nu)=\frac{1}{\omega(q)}.\]
\end{theorem}

The exact value of the Poincar\'{e} series of finite and affine Coxeter systems, and thus of Euclidean buildings, can be computed as follows. 

\begin{theorem}[Chevalley-Solomon~\cite{bjoerner}*{Theorem~7.1.5 on p.~204}]\label{thm: chevalley-solomon}
Let $(W, S)$ be a finite irreducible Coxeter system. There there positive integers $e_1, \ldots, e_d$, called \emph{exponents}, such that $d=|S|$ and 
	\[ \omega(t)=\prod_{i=1}^d (1+t+\ldots + t^{e_i}).\]
\end{theorem}

\begin{theorem}[Bott~\citelist{\cite{bott}\cite{bjoerner}*{Theorem~7.1.10 on p.~208}}]\label{thm: bott}
	Let $(W, S)$ be an affine irreducible Coxeter system. There there positive integers $e_1, \ldots, e_d$, which coincide with the exponents of the corresponding finite Coxeter system, such that 
	\[ \omega(t)=\prod_{i=1}^d \frac{1+t+\ldots + t^{e_i}}{1-t^{e_i}}.\]
\end{theorem}

The exponents are 
explicitly computed. For a reference for the following table see \emph{e.g.} the book by Bj\"orner and Brenti~\cite{bjoerner}*{Appendix~A1}.

\[\begin{array}{l|l} & e_1,e_2,\dots, e_d\\\hline
A_d, \mathrm{\tilde A}_d & 1,2,\dots,d\\
B_d, \mathrm{\tilde B}_d & 1,3,\dots,2d-1\\
C_d, \mathrm{\tilde C}_d & 1,3,\dots,2d-1\\
D_d, \mathrm{\tilde D}_d & 1,3,\dots,2d-3,d-1\\
G_2, \mathrm{\tilde G}_2 & 1,5\\
F_4, \mathrm{\tilde F}_4 & 1,5,7,11\\
E_6, \mathrm{\tilde E}_6 & 1,4,5,7,8,11\\
E_7, \mathrm{\tilde E}_7 & 1,5,7,9,11,13,17\\
E_8, \mathrm{\tilde E}_8 & 1,7,11,13,17,19,23,29
\end{array}
\]

In the remainder of Subsection~\ref{sub: chevalley} we consider the case where $$G=\bG(\F_q((t^{-1})))$$ where $\bG$ is a simple,  
simply connected Chevalley group over $\F_q$. Here, $\F_q((t^{-1}))$ denotes the locally compact field arising as the completion of $\F_q[t,t^{-1}]$ in the negative direction. This is a special case of Setup~\ref{setup}. The Bruhat-Tits building of $G$ has uniform thickness~$q$. If 
$\bG=\mathbf{SL_d}$ then the Bruhat-Tits building is of type $\mathrm{\tilde A}_{d-1}$ and is $(d-1)$-dimensional. 

Let $\bB<\bG$ be the Borel subgroup. Let 
\begin{equation}\label{eq: congruence}
\xi_l\colon \bG\bigl(\F_q[[t^{-1}]]\bigr)\to \bG\bigl(\F_q[[t^{-1}]]/t^{-l}\F_q[[t^{-1}]]\bigr)
\end{equation}
be the reduction map mod $t^{-l}$. 
The preimage $B$ of $\bB(\F_q)$ under $\xi_1$ is called the \emph{Iwahori subgroup} of $G$; the subgroup $B$ is the chamber stabiliser of the associated Bruhat-Tits building. In line with Setup~\ref{setup} we assume that the Haar measure $\nu$ of $G$ is normalised so that $\nu(B)=1$. 

\begin{theorem}\label{thm: computation for lattices}
Let $\bG$ be a simple, simply connected Chevalley group over $\F_q$. 
Let $d$ be the rank of $\bG$. 
Using the above normalisation of~$\nu$, we have 
\[ \betti_d(\bG(\F_q((t^{-1}))), \nu)=
	\prod_{i=1}^{d} \frac{q^{e_i}-1}{1+q+\ldots +q^{e_i}}>0.
\]
All other $L^2$-Betti numbers vanish. The only non-vanishing $L^2$-Betti number 
of the lattice $\bG(\F_q[t])<\bG(\F_q((t^{-1})))$ is 
\[ \betti_d\bigl(\bG(\F_q[t])\bigr)=
	\prod_{i=1}^{d} (q^{e_i+1}-1)^{-1}.\]

If we assume, in addition, that $\bG$ is of classical type and $d>1$ or of type $\mathrm{\tilde E_6}$, and $q=p^a>9$ is a power of a prime $p>5$ if $\bG$ is not of type $\mathrm{\tilde A}_d$, then the only non-vanishing $L^2$-Betti number of an  arbitrary lattice $\Gamma<\bG(\F_q((t^{-1})))$ 
is in degree $d$ and satisfies 
\[ \betti_d(\Gamma)\ge \betti_d\left(\bG(\F_q[t])\right).\]
\end{theorem}

\begin{proof}
By~\cite{davisetal}*{Corollary~14.5} the $L^2$-Betti numbers of 
$G=\bG(\F_q((t^{-1})))$ vanish except in the top dimension~$d$. 
Hence 
\[ \betti_{d}(G,\nu)=(-1)^d\chi(G,\nu).\]
Then the first statement follows from Theorems~\ref{thm: euler char} 
and~\ref{thm: bott} and the table above. 
Prasad's computation of covolumes~\cite{prasad} (see Golsefidy's account~\cite{golsefidy}*{\S 4} for the determination of Prasad's parameters in our case) yields 
\begin{equation}\label{eq: covolume formula}
\covol_\nu(\bG(\F_q[t]))=\frac{1}{|\bB(\F_q)|}\cdot\prod_{i=1}^{d}\frac{1}{1-q^{-e_i}}. 
\end{equation}
Let us explain the additional factor $|\bB(\F_q)|^{-1}$ which does not appear 
in~\cite{golsefidy}*{\S 4}: Golsefidy normalises the Haar measure so that the maximal normal pro-$p$ subgroup $\ker(\xi_1)$ (the first principal congruence subgroup), which is contained in the Iwahori subgroup~$B$, has measure~$1$. Since 
\[\xi_1\colon B/\ker(\xi_1)\xrightarrow{\cong} \bB(\F_q)\] 
is an isomorphism our normalisation of the Haar measure introduces the additional 
factor $|\bB(\F_q)|^{-1}$ into the formula from~\cite{golsefidy}*{\S 4}. 
From~\eqref{eq: lattice and G l2 Betti} we obtain that 
\begin{align}\label{eq: betti}
\betti_d\bigl(\bG(\F_q[t])\bigr) &= \covol_\nu(\bG(\F_q[t])\cdot \betti_d\bigl(\bG(\F_q((t^{-1})))\bigr)\\
&= \frac{1}{|\bB(\F_q)|}\cdot\prod_{i=1}^{d}\frac{q^{e_i}-1}{(1+q+\dots +q^{e_i})(1-q^{-e_i})}\\
&= \frac{1}{|\bB(\F_q)|}\cdot\prod_{i=1}^{d}\frac{q^{e_i}}{1+q+\ldots +q^{e_i}}.
\end{align}
Thus the second statement would follow from 
\begin{equation}\label{eq: number of roots}
	 |\bB(\F_q)|=(q-1)^dq^{e_1+\dots+e_d}.
\end{equation}
Let $m$ be the number of positive roots. According to~\cite{carter}*{Theorem~9.3.4 on p.~133}, 
\[ m=e_1+\dots+e_d.\]
The group $\bB(\F_q)$ is the semidirect product of the split torus $T$ and a maximal unipotent 
subgroup $U$. The first is isomorphic to $(\F_q^\times)^d$, the latter possesses a filtration with quotients isomorphic to $\F_q$ indexed by the positive roots. Hence the size of $\bB(\F_q)$ is $(q-1)^dq^m$, 
implying~\eqref{eq: number of roots}.

Finally, Golsefidy~\cite{golsefidy} proves that $\bG(\F_q[t])<G$ is a lattice of 
minimal covolume under our assumptions on the type and on $q$, which implies the inequality. 
\end{proof}

\subsection{Groups with weakly normal open amenable subgroups}\label{sub: neretin}
We now turn to a very different class of totally disconnected groups which includes the Neretin group.  
Here we obtain a vanishing result which generalizes~\cite{bader+furman+sauer}*{Theorem~1.3}.

\begin{theorem}
	Let $O<G$ be an open amenable subgroup such that for any finite sequence $g_1,\ldots, g_n$ of 
	elements in $G$ the intersection of conjugates $O^{g_1}\cap \ldots \cap O^{g_n}$ is noncompact. 
	Then all $L^2$-Betti numbers of $G$ vanish. 
\end{theorem}

\begin{proof}
Consider the simplicial complex $X$ whose set of $n$-simplices is the $(n+1)$-fold 
product $G/O\times\ldots\times G/O$ with the obvious projections as face maps. There is a natural diagonal 
$G$-action on $X$ which makes $X$ into a smooth $G$-CW-complex. In particular, 
$X$ is a CW-complex with a cellular $G$-action. The space $X$ is contractible. The 
stabiliser of an $p$-cell $\sigma$ is 
an intersection of conjugates of $O$ 
\[ G_\sigma=O^{g_1}\cap \ldots \cap O^{g_n}\] 
which is a noncompact amenable subgroup. 
Let $P_\ast$ be a $\cH(G)$-resolution by projective right modules of the trivial module~$\C$. 
Let $C_\ast=C_\ast(X)\otimes_{\C}L^2_\infty(G)$ where $C_\ast(X)$ is 
the cellular chain complex of $X$ with $\C$-coefficients. With the obvious diagonal 
$G$-action $C_\ast$ becomes a chain complex of discrete $G$-modules. 
By Remark~\ref{rem: discrete G modules} $C_\ast$ thus becomes a chain complex 
of $\cH(G)$-modules. We can easily identify this module structure via the 
isomorphisms 
\[ C_p\cong\bigoplus_\sigma\C[G_\sigma\bs G]\otimes_{\C} L^2_\infty(G)\cong 
\bigoplus_\sigma \cH(G)\otimes_{\cH(G_\sigma)} L^2_\infty(G)\]
where $\sigma$ runs through representatives of orbits of $p$-cells. The right hand side has an obvious left  
$\cH(G)$-module structure, and the isomorphism between $C_p$ and the right hand side is $\cH(G)$-linear. 

Now consider the double complex 
$D_{\ast\ast}:=P_\ast\otimes_{\cH(G)} C_\ast$. By well known homological algebra one obtains 
two spectral sequences converging to the total homology~\cite[Chapter VII.7]{brown}. Since tensoring with a projective module is flat and $X$ is contractible, the $E^1$-term 
of the first one is 
\[E^1_{p,q}=H_q(P_\ast\otimes_{\cH(G)} C_\ast)=P_p\otimes_{\cH(G)} H_q(C_\ast)=\begin{cases}
	P_p\otimes_{\cH(G)}L^2_\infty(G)&\text{ if $q=0$,}\\
	0 & \text{ otherwise.}
\end{cases}\]
So the spectral sequence converges to $H_p(G,L^2_\infty(G))$. The second spectral sequence has 
the $E^1$-term 
\[E^1_{p,q}=H_p(P_\ast\otimes_{\cH(G)}C_q)\cong H_p(G, \bigoplus_\sigma \cH(G)\otimes_{\cH(G_\sigma)} L^2_\infty(G))\cong 
\bigoplus_{\sigma} H_p(G_\sigma, L^2_\infty(G)). 
\]
The last isomorphism is the Shapiro isomorphism. Since $G_\sigma$ is amenable and noncompact, the dimension 
of every $E^1_{p,q}$ vanishes~\cite{Pe11}*{Theorem~7.10}. Hence also the dimension of $H_\ast(G, L^2(G))$, i.e.~the $L^2$-Betti numbers 
of $G$ vanish. 
\end{proof}

The Neretin group is an interesting totally disconnected group to which we can apply the previous theorem. The required open amenable 
subgroup is the group $O$ introduced after Theorem~1.1 in~\cite{bader+caprace+gelander+mozes}. It is shown in \emph{loc.\,cit.} that the Neretin group is a simple locally compact group without any lattices. Moreover, the Neretin group admits a topological model with 
cocompact $n$-skeleton for every $n\ge 0$~\cite{sauer+thumann}. 
The following corollary solves Problem~1.1 in~\cite{Pe11}. 

\begin{corollary}
	The $L^2$-Betti numbers of the Neretin group vanish. 
\end{corollary}

\subsection{Low degree Morse inequalities}

In this section we provide an upper bound for the first $L^2$-Betti number of a compactly generated group. 

The correct analogue of the Cayley graph for a compactly generated totally disconnected group~$G$ is the Cayley-Abels graph. The Cayley-Abels graph also arises as the $1$-skeleton of a suitable topological model of $G$. Let $K \subset G$ be a compact and open subgroup. Let $S_0 \subset G$ compact so that $S = K S_0 K$ is a compact generating set of $G$. For example, $S_0$ can be a generating set of a dense subgroup of $G$.
The \emph{Cayley-Abels graph} ${\rm Cay}(G,K,S)$ is the quotient of the Cayley graph of $(G,S)$ (edges defined on the left) by the left $K$-action. 
So the vertices are cosets $Kg$ and there is an edge (but these are no longer labelled) from $Kg$ to $Ksg$.

Next we record some familiar properties of the Cayley-Abels graph ${\rm Cay}(G,K,S)$ together with its right $G$-action. First of all, ${\rm Cay}(G,K,S)$ has one orbit $K \bs G$ of vertices. The degree of ${\rm Cay}(G,K,S)$ is $K\bs S$. But there is a difference to the discrete case, where the degree of any vertex is exactly the number of edge-orbits of the action. Here, it is possible that the edge $(Ke,Ks)$ and $(Ke,Ksk_0)$ are translates of each other. The orbit of edges correspond to double cosets $K \bs S /K$, whose number might not be the degree.
The stabiliser of the edge $(Ke,Ks)$ is equal to $K \cap s^{-1}Ks = K \cap K^s$.
Thus, as a $G$-set, the edges can be identified with 
\[\bigsqcup_{KsK \in K \bs S /K} (K \cap K^s) \bs G.\]
Thus we obtain an exact sequence
\[\bigoplus_{KsK \in K \bs S /K} \C[(K \cap K^s) \bs G]  \to \C[K \bs G] \to \C \to 0\]
of right $\cH(G)$-modules. This can be used to estimate the low-degree $L^2$-Betti numbers. Indeed if $G$ is not compact, then $\beta_0^{(2)}(G)=0$. Together with $[K:K \cap K^s] \cdot \nu(K \cap K^s) = \nu(K)$, we obtain in this case that 
\[\beta_1^{(2)}(G,\nu) \leq \frac1{\nu(K)} \cdot \left(\sum_{KsK \in K \bs S /K} [K:K \cap K^s]- 1 \right).\]
This formula is well known in the case where $G$ is discrete, $K=\{1\}$, and $\nu$ is just the counting measure. But already when $G$ is discrete and $K \subset G$ is some finite subgroup, we obtain new upper bounds on the first $L^2$-Betti number of $G$. In the non-discrete case, the estimate relates the number of generators and the distortion of a compact-open subgroup $K$ by the generators to the first $L^2$-Betti number.

There are analogous lower and upper bounds for combinations of higher $L^2$-Betti numbers that arise from the structure of the topological model. Indeed, if $G$ admits a topological model with finite $n$-skeleton, then the usual Morse inequalities take the form
$$(-1)^n\sum_{i=0}^n (-1)^i \beta_i^{(2)}(G) \leq (-1)^n\sum_{i=0}^{n} (-1)^{i} \sum_{U\in \cF_i} \mu(U)^{-1},$$
where $\cF_i$ is the set of $i$-cells and $U$ the corresponding stabilizer. In particular, we can use these inequalities easily to derive bounds for low degree $L^2$-Betti numbers.

\vspace{0.1cm}

Let us just mention one concrete application concerning the existence of central extensions of lattices.
\begin{corollary}
Let $G$ be a unimodular, totally disconnected group and let $(\Gamma_i)_{i \in \mathbb N}$ be a Farber sequence of lattices. If $\beta^{(2)}_2(G,\nu)>0$, then $\Gamma_i$ admits a non-trivial central extension for $i$ large enough.
\end{corollary}
\begin{proof}
This is a straightforward application of Theorem \ref{thm:main}.
\end{proof}

\subsection{$L^2$-Betti numbers of locally compact groups and their lattices}\label{sub: l2 betti and their lattices}

In this subsection we provide a short proof of~\eqref{eq: lattice and G l2 Betti}. 
Let $C_* \to \C$ be a projective resolution of the trivial $\cH(G)$-module. Each projective $\cH(G)$-module is a filtered limit of modules of the form $\C[K \bs G]$ where $K<G$ is a compact open subgroup (cf.~Section~\ref{sub: homological algebra}). Since 
\[\C[K \bs G] = \bigoplus_{Ks\Gamma \in K \bs G/\Gamma} \C[(K^s \cap \Gamma) \bs \Gamma]\]
as a right $\C[\Gamma]$-module, $\C[K \bs G]$ is projective as a right $\C[\Gamma]$-module. Note that it is a finitely generated projective module if and only if $K \bs G / \Gamma$ is finite, i.e., if and only if $\Gamma$ is cocompact. 
%It is worth noticing, that we obtain an inclusion
%$$\cH(G,K) = {\rm End}_{G}(\C[K\bs G]) \subset {\rm End}_{\Gamma}(\C[K\bs G]),$$
%i.e., we can identify the Hecke algebra of the Hecke pair $(G,K)$ with a unital $*$-subring of the endomorphism ring of a projective right $\C[\Gamma]$-module. 
The natural inclusion
\begin{align*}
\C[K \bs G] \otimes_{\C[\Gamma]} \ell^2 (\Gamma) &= \bigoplus_{Ks\Gamma \in K \bs G/\Gamma} \C[(K^s \cap \Gamma) \bs \Gamma] \otimes_{\C[\Gamma]} \ell^2(\Gamma)\\
&= \bigoplus_{Ks\Gamma \in K \bs G/\Gamma} \ell^2((K^s \cap \Gamma) \bs \Gamma)\\
& \subset \ell^2\Bigl(\bigsqcup_{Ks\Gamma \in K \bs G/\Gamma} (K^s \cap \Gamma) \bs \Gamma \Bigr) \\
&=\ell^2(K\bs G) =\C[K \bs G] \otimes_{\cH(G)} L^{2}_{\infty}(G)
\end{align*}
is a $L(\Gamma)$-dimension isomorphism since the algebraic sum is rank dense in the $\ell^2$-sum above.
The latter follows from the fact that $\Gamma$ has finite covolume and hence 
\[\sum_{Ks\Gamma \in K \bs G/\Gamma} \dim_{L(\Gamma)} \ell^2((K^s \cap \Gamma) \bs \Gamma)=\sum_{Ks\Gamma \in K \bs G/\Gamma} \frac1{|K^s \cap \Gamma|} < \infty.\]
Thus, the natural map
\[H_*(C_* \otimes_{\C[\Gamma]} \ell^2(\Gamma)) \to H_*\left(C_* \otimes_{\cH(G)} L^2_{\infty}(G) \right)\]
is a $L(\Gamma)$-dimension isomorphism (one uses exactness of rank completion~\cite{thom-rank} and the local criterion~\citelist{\cite{thom-rank}*{Theorem~1.2}\cite{sauer-groupoids}*{Theorem~2.4}}). 
Since $\dim_{L(\Gamma)} (M) = {\rm covol}_{\nu}(\Gamma) \cdot \dim_{(L(G),{\rm tr})} (M)$ for any right $L(G)$-module, we obtain the desired equality~\eqref{eq: lattice and G l2 Betti}. 

\subsection{The Connes Embedding Problem for Hecke-von Neumann algebras}

The Connes Embedding Problem is a major open problem in the theory of von Neumann algebras. It asserts that every finite von Neumann algebra with a separable pre-dual can be embedded into an von Neumann algebraic ultraproduct of the hyperfinite II$_1$-factor. We call finite von Neumann algebras for which such an embedding exists \emph{embeddable}. Loosely speaking, a finite von Neumann algebra is embeddable if the joint moments of any finite subset of elements can be approximated by the joint moments of complex matrices.

Typical examples of finite von Neumann algebras are group von Neumann algebras. Then, the embeddability of group von Neumann algebras is related to other famous open problems, such as Gromov's question whether all discrete groups are sofic, see Pestov's survey for more information on these notions \cite{pestov}.

Another source of finite von Neumann algebras are so-called Hecke-von Neumann algebras of Hecke pairs $(G,K)$. Let $G$ be a locally compact, totally disconnected, second countable, unimodular group with Haar measure $\mu$ and let $K \subset G$ be a compact open subgroup. As before, we denote by $L(G) \subset B(L^2(G,\mu))$ the group von Neumann algebra and by $L(G,K) = p_K L(G) p_K$ the Hecke-von Neumann algebra associated with the Hecke pair $(G,K)$. As we have seen, the Hecke-von Neumann algebra admits a unital faithful and positive normal trace, and thus it is a finite von Neumann algebra. It is natural to ask under which circumstances we are able to give a positive answer to the Connes Embedding Problem. Our main result in this direction is the following theorem.

\begin{theorem}
Let $G$ be a locally compact, totally disconnected, second countable, unimodular group and let $K \subset G$ be a compact open subgroup. If $G$ admits a Farber sequence $(\Gamma_i)_{i \in \mathbb N}$, then the finite von Neumann algebra $L(G,K)$ is embeddable.
\end{theorem}
\begin{proof}
The proof is a side-product of the techniques that were developed in order to prove Theorem \ref{thm:spectral_appr}. Indeed, for any $T \in \cH(G,K) \subset L(G,K)$, we consider the sequence of matrices $P_{\Gamma_i}\pi_i(T)P_{\Gamma_i} \in B(\ell^2(\bar X_{e,i}))$. Note that the operator norm of $P_{\Gamma_i}\pi_i(T)P_{\Gamma_i}$ is bounded independent of $i \in \mathbb N$. It remains to show that for any finite list $T_1,\dots,T_k \in \cH(G,K)$ the joint moments of $P_{\Gamma_i}\pi_i(T_1)P_{\Gamma_i},\dots,P_{\Gamma_i}\pi_i(T_k)P_{\Gamma_i}$ converge to the joint moments of $T_1,\dots,T_k$ as the parameter $i$ tends to infinity. However, this is obvious from the constructions in the proof of Theorem \ref{thm:spectral_appr}. 
Now, this implies that
$$\mathcal H(G,K) \ni T \mapsto (P_{\Gamma_i}\pi_i(T)P_{\Gamma_i})_{i \in \mathbb N} \in \prod_{i \in \mathbb N}B\left(\ell^2(\bar X_{e,i})\right)$$
induces a trace-preserving embedding of $\mathcal H(G,K)$ into an von Neumann algebraic ultra-product of matrix algebras. This finishes the proof. 
\end{proof}

\appendix

\section{Invariant random subgroups}\label{sec: chabauty}

The set $\sub_G$ of closed subgroups of a locally compact group $G$ carries the \emph{Chabauty topology} which is generated by two types of sets, namely  
    \begin{align*}
       O_1(C)&=\{H\in\sub_G\mid H\cap C=\emptyset\},~~\text{$C\subset G$ compact,}\\
	   O_2(U)&=\{H\in\sub_G\mid H\cap U\ne\emptyset\},~~\text{$U\subset G$ open.} 
	\end{align*}

Let us now justify the results that we claimed in the introduction. Proposition~\ref{prop: on Farber vs IRS} in the introduction is a consequence of the following lemma. 

\begin{lemma}
	A sequence $(\Gamma_i)_{i \in \mathbb N}$ of lattices in $G$ converges as invariant random subgroups to the trivial invariant random subgroup $\delta_e$ if and only if for every compact open subgroup $K<G$ and for every right coset $C\ne K$ of $K$ the probability that a conjugate of $g\Gamma_ig^{-1}$ meets $C$ tends to zero for $i\to\infty$. 
\end{lemma}
	
\begin{proof}
    First we show the $\Leftarrow$-statement. 
	By the Portmanteau theorem weak convergence $\mu_{\Gamma_i}\to \delta_{\{e\}}$ 
	is equivalent to $\limsup_{i\to\infty}\mu_{\Gamma_i}(F)\le \delta_{\{e\}}(F)$ for 
	every closed $F\subset\sub_G$, and equivalent to $\liminf_{i\to\infty}\mu_{\Gamma_i}(V)\ge \delta_{\{e\}}(V)$ for 
	every open $V\subset \sub_G$. Let $K<G$ be a compact open subgroup. 
	It suffices to show that 
	\begin{equation}\label{eq: inequality in weak convergence}
	\liminf_{i\to\infty}\mu_{\Gamma_i}(V)\ge \delta_{\{e\}}(V) 
	\end{equation}
	for the elements $V$ of a subbasis of the Chabauty topology. To show~\eqref{eq: inequality in weak convergence} we may assume that $\{e\}\in V$. Consider first the case $V=O_2(U)$ for some 
	open $U\subset G$. Then $\{e\}\in V$ means 
	$e\in U$ and hence $\mu_{\Gamma_i}(V)=1$ for all $i\in\N$. Now let $V=O_1(C)$ for some compact $C\subset G$. 
	Then $\{e\}\in V$ means $e\not\in C$. 
	We can find a finite number of cosets $h_kK\ne K$ such that $C$ is contained in their union. By the assumption, 
		\[\lim_{i\to\infty} \mu_{\Gamma_i}(O_1(h_kK))=1\] 
	for every $k$. Hence 
		\[\lim_{i\to\infty} \mu_{\Gamma_i}\bigl(\bigcap_k O_1(h_kK)\bigr)=1\] 
	for the finite intersection. Since $\bigcap_k O_1(h_kK)$ is contained in $V$ we obtain~\eqref{eq: inequality in weak convergence} in this case. 
		
	Next we show the $\Rightarrow$-direction. Assume that $\mu_{\Gamma_i}\to \delta_{\{e\}}$ in $\irs_G$. 
	Let $K<G$ be a compact and open subgroup. For every coset $hK\ne K$ 
	the subset $O_2(hK)$ is closed and open in the Chabauty topology. 
	Let $h\not\in K$. From $O_2(hK)$ being closed and weak convergence we deduce that 
	\[ \limsup_{i\to\infty}\mu_{\Gamma_i}\bigl(\{H\mid H\cap hK\ne\emptyset\}\bigr)=\limsup_{i\to\infty} \mu_{\Gamma_i}(O_2(hK))\le \delta_{\{e\}}(O_2(hK))=0.\qedhere\]
\end{proof}

\begin{proof}[Proof of Proposition~\ref{prop: on Farber vs IRS}]
Suppose that the sequence $(\Gamma_i)$ is weakly uniformly discrete and 
$\mu_{\Gamma_i}\to \delta_{\{e\}}$. By the previous lemma it suffices to show that 
for a compact-open subgroup $K<G$ we have 
$\mu_{\Gamma_i}(\{H\mid H\cap K\ne \{e\}\})\to 0$ for $i\to \infty$. Suppose this is not true. Then there is $\epsilon>0$ such that upon passing to a subsequence of lattices 
\[ \mu_{\Gamma_i}(\{H\mid H\cap K\ne \{e\}\})>\epsilon\]
for all $i\in\N$. By weak uniform discreteness there is a compact-open subgroup 
$K_1<G$ such that  
\[ \mu_{\Gamma_i}(\{H\mid H\cap K_1\ne \{e\}\})<\epsilon/2\]
for all $i\in\N$. This implies that 
\[ \mu_{\Gamma_i}(\{H\mid H\cap K\cap (G\backslash K_1)\ne \emptyset\})\ge \epsilon/2\]
for all $i\in\N$ which contradicts the assumption that $\mu_{\Gamma_i}\to \delta_{\{e\}}$ since $K\cap (G\backslash K_1)$ is closed and $\delta_{\{e\}}(K\cap (G\backslash K_1))=0$. 

Next suppose that $(\Gamma_i)$ is a Farber sequence. The previous lemma immediately implies that $\mu_{\Gamma_i}\to \delta_{\{e\}}$. It suffices to show that 
$(\Gamma_i)$ is weakly uniformly discrete. Let $\epsilon>0$. Let $K<G$ be a compact-open subgroup. Then there is $i_0\in\N$ such that 
\[\mu_{\Gamma_i}(\{H\mid H\cap K\ne \{e\}\})<\epsilon\]
for $i>i_0$. By discreteness of each $\Gamma_i$ there is a compact-open subgroup 
$K_1<G$ such that $\mu_{\Gamma_i}(\{H\mid H\cap K_1\ne \{e\}\})<\epsilon$ for 
every $i\in\{1,\dots, i_0\}$. So the compact-open subgroup $K\cap K_1$ satisfies $\mu_{\Gamma_i}(\{H\mid H\cap (K\cap K_1)\ne \{e\}\})<\epsilon$ for every $i\in\N$. 
\end{proof}

The following lemma clarifies why our main result is a generalisation of L\"uck's approximation theorem.

\begin{lemma}\label{lem: residual chain is farber}
Let $\Gamma$ be a lattice in $G$ and $(\Gamma_i)_{i \in \mathbb N}$ be a chain of finite-index normal subgroups of $\Gamma$ such that $\bigcap_{i \in \mathbb N} \Gamma_i = \{e\}$. Then, $(\Gamma_i)_{i \in \mathbb N}$ is a Farber sequence, seen as a sequence of lattices in $G$.
\end{lemma}
\begin{proof} Let $K$ be a compact open subgroup of $G$. Note that since $\Gamma_i \subseteq \Gamma$ is normal, conjugation of $\Gamma_i$ by $g\Gamma \in G/\Gamma$ is well-defined and allows to understand the behaviour of random conjugates of $\Gamma_i$.
Consider the sequence of measurable subsets $A_i \subseteq  G/\Gamma$, defined as
$$A_i := \{g\Gamma \in G/\Gamma \mid g\Gamma_i g^{-1} \cap K = \{e\} \}.$$ Then, $A_i$ is a monotone increasing sequence of sets and by our assumption, the union of the $A_i$'s is equal to $G/\Gamma$. Now, for every $\varepsilon>0$, by $\sigma$-additivity of the Haar measure, there exists some $n \in \mathbb N$ such that for all $i \geq n$ we have $\nu(A_i) \geq (1- \varepsilon)\covol_{\nu}(\Gamma)$. Thus, $(\Gamma_i)_{i \in \mathbb N}$ is a Farber sequence. This finishes the proof.
\end{proof}

%
%	If $A_H=\{H\mid H\cap K\ne\{e\}\}$ is closed in the Chabauty topology, then the Farber condition for the coset $C=K$ in Definition~\ref{defn: farber condition} would follow from  
%	weak convergence $\mu_{\Gamma_i}\to \delta_{\{e\}}$. Hence in this case 
%	the Farber condition is equivalent to $\mu_{\Gamma_i}\to \delta_{\{e\}}$. We remark that $A_H$ is Chabauty closed 
%	if $K\subset G$ is any compact subset and $G$ is a Lie group. 
%	This is based on the observation that if $(g_i)_{i \in \mathbb N}$, $g_i\ne e$, is a bounded sequence in a Lie group $G$, then there is a sequence $(n_i)_{i \in \mathbb N}$ of 
%	integers so that $(g_i^{n_i})_{i \in \mathbb N}$ converges, upon passing to 
%	a subsequence, to an element $g\in G\backslash\{1\}$.

For semisimple algebraic groups there might be no difference between 
convergence as invariant random subgroups and the property of being a Farber sequence. See Remark~\ref{rem: weak uniform discreteness} in connection 
with Proposition~\ref{prop: on Farber vs IRS}. However, for arbitrary totally disconnected groups there is a difference: There are easy examples of a sequence of lattices $(\Gamma_i)_{i \in \mathbb N}$ with increasing covolumes, which converges to the trivial subgroup as a sequence of invariant random subgroups but do not form a Farber sequence. Indeed, just take $G= \prod_{n \in \mathbb N} (\mathbb Z/2 \mathbb Z)$ and $\Gamma_i = \prod_{n=i}^{2i} (\mathbb Z/2 \mathbb Z)$.

\section*{Acknowledgments}
We thank the referee for a simplification in the proof of Theorem~\ref{thm: computation for lattices}. 

First results were obtained during a visit of H.P. in Leipzig in 2012 -- he thanks University of Leipzig for its hospitality. The authors presented some of the results of this paper at a Spring School in 2015 in Sde Boker -- we greatly enjoyed this meeting and thank the organisers for their encouragement. R.S. thanks Uri Bader and Adrien le Boudec for discussions on Section~\ref{sub: neretin} and Arie Levit for discussions on Section~\ref{sec: chabauty}. We thank Jean Raimbault and Damien Gaboriau for helpful comments on a previous version of this manuscript. A.~Thom was supported by ERC Starting Grant No.\ 277728 and ERC Consolidator Grant No.\ 681207, and R.~Sauer was supported by DFG Grant 1661/4.


\begin{bibdiv} 
\begin{biblist}

\bib{7sam_announce}{article}{
   author={Ab\'ert, Miklos},
   author={Bergeron, Nicolas},
   author={Biringer, Ian},
   author={Gelander, Tsachik},
   author={Nikolov, Nikolay},
   author={Raimbault, Jean},
   author={Samet, Iddo},
   title={On the growth of Betti numbers of locally symmetric spaces},
   language={English, with English and French summaries},
   journal={C. R. Math. Acad. Sci. Paris},
   volume={349},
   date={2011},
   number={15-16},
   pages={831--835},
}

\bib{7sam}{article}{
   author={Ab\'ert, Miklos},
   author={Bergeron, Nicolas},
   author={Biringer, Ian},
   author={Gelander, Tsachik},
   author={Nikolov, Nikolay},
   author={Raimbault, Jean},
   author={Samet, Iddo},
   title={On the growth of $L^2$-invariants for sequences of lattices in Lie groups},
   date={2012},
   note={ArXiv~1210.2961}	
 }
 \bib{bader+caprace+gelander+mozes}{article}{
   author={Bader, Uri},
   author={Caprace, Pierre-Emmanuel},
   author={Gelander, Tsachik},
   author={Mozes, Shahar},
   title={Simple groups without lattices},
   journal={Bull. Lond. Math. Soc.},
   volume={44},
   date={2012},
   number={1},
   pages={55--67},
   %issn={0024-6093},
   %review={\MR{2881324}},
   %doi={10.1112/blms/bdr061},
}
\bib{bader+furman+sauer}{article}{
   author={Bader, Uri},
   author={Furman, Alex},
   author={Sauer, Roman},
   title={Weak notions of normality and vanishing up to rank in $L^2$-cohomology},
   journal={Int. Math. Res. Not. IMRN},
   date={2014},
   number={12},
   pages={3177--3189},
   %issn={1073-7928},
   %review={\MR{3217658}},
}

\bib{MR2059438}{article}{
   author={Bergeron, N.},
   author={Gaboriau, D.},
   title={Asymptotique des nombres de Betti, invariants $l^2$ et
   laminations},
   language={French, with English summary},
   journal={Comment. Math. Helv.},
   volume={79},
   date={2004},
   number={2},
   pages={362--395},
   issn={0010-2571},
}
\bib{bjoerner}{book}{
   author={Bj{\"o}rner, Anders},
   author={Brenti, Francesco},
   title={Combinatorics of Coxeter groups},
   series={Graduate Texts in Mathematics},
   volume={231},
   publisher={Springer, New York},
   date={2005},
   pages={xiv+363},
   %isbn={978-3540-442387},
   %isbn={3-540-44238-3},
   %review={\MR{2133266}},
}

%\bib{borel}{article}{
%   author={Borel, Armand},
%   title={Cohomologie de certains groupes discretes et laplacien $p$-adique
%   (d'apr\`es H. Garland)},
%   language={French},
%   conference={
%      title={S\'eminaire Bourbaki, 26e ann\'ee (1973/1974), Exp. No. 437},
%   },
%   book={
%      publisher={Springer, Berlin},
%   },
%   date={1975},
%   pages={12--35. Lecture Notes in Math., Vol. 431},
%   %review={\MR{0476919}},
%}
\bib{borel+wallach}{book}{
    AUTHOR = {Borel, Armand},
    author={Wallach, Nolan R.},
     TITLE = {Continuous cohomology, discrete subgroups, and representations
              of reductive groups},
    SERIES = {Annals of Mathematics Studies},
    VOLUME = {94},
 PUBLISHER = {Princeton University Press},
   ADDRESS = {Princeton, N.J.},
      YEAR = {1980},
     PAGES = {xvii+388},
      }

\bib{bott}{incollection}{
   author={Bott, Raoul},
   title={The geometry and representation theory of compact Lie groups},
   series={London Mathematical Society Lecture Note Series},
   volume={34},
   booktitle={Proceedings of the SRC/LMS Research Symposium held in Oxford,
   June 28--July 15, 1977},
   editor={Luke, G. L.},
   publisher={Cambridge University Press, Cambridge-New York},
   date={1979},
   %pages={v+341},
}


\bib{brown}{book}{
   author={Brown, Kenneth S.},
   title={Cohomology of groups},
   series={Graduate Texts in Mathematics},
   volume={87},
   note={Corrected reprint of the 1982 original},
   publisher={Springer-Verlag, New York},
   date={1994},
   pages={x+306},
  }
 \bib{bux}{article}{
   author={Bux, Kai-Uwe},
   author={Wortman, Kevin},
   title={Finiteness properties of arithmetic groups over function fields},
   journal={Invent. Math.},
   volume={167},
   date={2007},
   number={2},
   pages={355--378},
   %issn={0020-9910},
   %review={\MR{2270455}},
   %doi={10.1007/s00222-006-0017-y},
}
		

\bib{carter}{book}{
   author={Carter, Roger W.},
   title={Simple groups of Lie type},
   series={Wiley Classics Library},
   note={Reprint of the 1972 original;
   A Wiley-Interscience Publication},
   publisher={John Wiley \& Sons, Inc., New York},
   date={1989},
   pages={x+335},
   %isbn={0-471-50683-4},
   %review={\MR{1013112}},
}



\bib{davisetal}{article}{
   author={Davis, Michael W.},
   author={Dymara, Jan},
   author={Januszkiewicz, Tadeusz},
   author={Okun, Boris},
   title={Weighted $L^2$-cohomology of Coxeter groups},
   journal={Geom. Topol.},
   volume={11},
   date={2007},
   pages={47--138},
}
\bib{davis-book}{book}{
   author={Davis, Michael W.},
   title={The geometry and topology of Coxeter groups},
   series={London Mathematical Society Monographs Series},
   volume={32},
   publisher={Princeton University Press, Princeton, NJ},
   date={2008},
   pages={xvi+584},
}

\bib{dymara}{article}{
   author={Dymara, Jan},
   title={Thin buildings},
   journal={Geom. Topol.},
   volume={10},
   date={2006},
   pages={667--694},
}
%\bib{dymara+janus}{article}{
%   author={Dymara, Jan},
%   author={Januszkiewicz, Tadeusz},
%   title={Cohomology of buildings and their automorphism groups},
%   journal={Invent. Math.},
%   volume={150},
%   date={2002},
%   number={3},
%   pages={579--627},
%   %issn={0020-9910},
%   %review={\MR{1946553}},
%   %doi={10.1007/s00222-002-0242-y},
%}
\bib{farber}{article}{
   author={Farber, Michael},
   title={Geometry of growth: approximation theorems for $L^2$
   invariants},
   journal={Math. Ann.},
   volume={311},
   date={1998},
   number={2},
   pages={335--375},
}

\bib{MR1953191}{article}{
   author={Gaboriau, D.},
   title={Invariants $l^2$ de relations d'\'equivalence et de groupes},
   language={French},
   journal={Publ. Math. Inst. Hautes \'Etudes Sci.},
   number={95},
   date={2002},
   pages={93--150},
   issn={0073-8301},
}

\bib{MR2221157}{article}{
   author={Gaboriau, D.},
   title={Invariant percolation and harmonic Dirichlet functions},
   journal={Geom. Funct. Anal.},
   volume={15},
   date={2005},
   number={5},
   pages={1004--1051},
   issn={1016-443X},
}

\bib{gelander}{article}{
	author={Gelander, T.}, 
	title={Kazhdan-Margulis theorem for invariant random subgroups},
	date={2016},
	note={Preprint, arXiv:1510.05423}
}

\bib{gelander+levit}{article}{
	author={Gelander, T.}, 
	author={Levit, A.}, 
	title={Invariant random subgroups in totally algebraic groups over non-Archimedean local fields},
	date={2015},
}
\bib{harder}{article}{
   author={Harder, G.},
   title={Die Kohomologie $S$-arithmetischer Gruppen \"uber Funktionenk\"orpern},
   language={German},
   journal={Invent. Math.},
   volume={42},
   date={1977},
   pages={135--175},
   %issn={0020-9910},
   %review={\MR{0473102}},
   %doi={10.1007/BF01389786},
}
				
\bib{kyed+petersen+vaes}{article}{
   author={Kyed, David},
   author={Petersen, Henrik Densing},
   author={Vaes, Stefaan},
   title={$L^2$-Betti numbers of locally compact groups and their cross
   section equivalence relations},
   journal={Trans. Amer. Math. Soc.},
   volume={367},
   date={2015},
   number={7},
   pages={4917--4956},
}

\bib{Lu94}{article}{
   author={L{\"u}ck, Wolfgang},
   title={Approximating $L^2$-invariants by their finite-dimensional
   analogues},
   journal={Geom. Funct. Anal.},
   volume={4},
   date={1994},
   number={4},
   pages={455--481},
}
\bib{lueck-survey}{article}{
   author={L{\"u}ck, Wolfgang},
   title={Survey on classifying spaces for families of subgroups},
   conference={
      title={Infinite groups: geometric, combinatorial and dynamical
      aspects},
   },
   book={
      series={Progr. Math.},
      volume={248},
      publisher={Birkh\"auser, Basel},
   },
   date={2005},
   pages={269--322}
}
\bib{lueck-book}{book}{
   author={L{\"u}ck, Wolfgang},
   title={$L^2$-invariants: theory and applications to geometry and
   $K$-theory},
   series={Ergebnisse der Mathematik und ihrer Grenzgebiete. 3. Folge. A
   Series of Modern Surveys in Mathematics [Results in Mathematics and
   Related Areas. 3rd Series. A Series of Modern Surveys in Mathematics]},
   volume={44},
   publisher={Springer-Verlag, Berlin},
   date={2002},
   pages={xvi+595},
}

\bib{pestov}{article}{
   author={Pestov, Vladimir G.},
   title={Hyperlinear and sofic groups: a brief guide},
   journal={Bull. Symbolic Logic},
   volume={14},
   date={2008},
   number={4},
   pages={449--480},
   issn={1079-8986},
}

\bib{Pe11}{thesis}{
   author={Petersen, Henrik Densing},
   title={$L^2$-Betti numbers of locally compact groups},
   status={Ph.D. thesis, Department of Mathematical Sciences, Faculty of Science, University of Copenhagen, 2012. 139 p.},
}

\bib{MR123}{article}{
   author={Petersen, Henrik Densing},
   title={$L^2$-Betti numbers of locally compact groups},
   language={English, with English and French summaries},
   journal={C. R. Math. Acad. Sci. Paris},
   volume={351},
   date={2013},
   number={9-10},
   pages={339--342},
}

\bib{MR3158720}{article}{
   author={Petersen, Henrik Densing},
   author={Valette, Alain},
   title={$L^2$-Betti numbers and Plancherel measure},
   journal={J. Funct. Anal.},
   volume={266},
   date={2014},
   number={5},
   pages={3156--3169},
}

\bib{MR2827095}{article}{
   author={Peterson, Jesse},
   author={Thom, Andreas},
   title={Group cocycles and the ring of affiliated operators},
   journal={Invent. Math.},
   volume={185},
   date={2011},
   number={3},
   pages={561--592},
}
\bib{prasad}{article}{
   author={Prasad, Gopal},
   title={Volumes of $S$-arithmetic quotients of semi-simple groups},
   note={With an appendix by Moshe Jarden and the author},
   journal={Inst. Hautes \'Etudes Sci. Publ. Math.},
   number={69},
   date={1989},
   pages={91--117},
   %issn={0073-8301},
   %review={\MR{1019962}},
}
\bib{golsefidy}{article}{
   author={Salehi Golsefidy, Alireza},
   title={Lattices of minimum covolume in Chevalley groups over local fields
   of positive characteristic},
   journal={Duke Math. J.},
   volume={146},
   date={2009},
   number={2},
   pages={227--251},
   %issn={0012-7094},
   %review={\MR{2477760}},
   %doi={10.1215/00127094-2008-064},
}

\bib{sauer-groupoids}{article}{
   author={Sauer, Roman},
   title={$L^2$-Betti numbers of discrete measured groupoids},
   journal={Internat. J. Algebra Comput.},
   volume={15},
   date={2005},
   number={5-6},
   pages={1169--1188},
}

\bib{sauer-survey}{article}{
   author={Sauer, Roman},
   title={$\ell^2$-Betti numbers of discrete and non-discrete groups},
   book={
          title={New directions in locally compact groups},
          series={London Mathematical Society Lecture Note Series},
          volume={447},
          publisher={Cambridge University Press},
          date={2018}
         },
   pages={205--226},
}



	
\bib{MR2572246}{article}{
   author={Sauer, Roman},
   title={Amenable covers, volume and $L^2$-Betti numbers of aspherical
   manifolds},
   journal={J. Reine Angew. Math.},
   volume={636},
   date={2009},
   pages={47--92},
}

\bib{sauer+schroedl}{article}{
   author={Sauer, Roman},
   author={Schr\"odl, Michael},
   title={Vanishing of $\ell^2$-Betti numbers of locally compact groups as an invariant of coarse equivalence},
   note={Preprint, arXiv:1702.01685}
   date={2017},
}



\bib{MR2650795}{article}{
   author={Sauer, Roman},
   author={Thom, Andreas},
   title={A spectral sequence to compute $L^2$-Betti numbers of groups
   and groupoids},
   journal={J. Lond. Math. Soc. (2)},
   volume={81},
   date={2010},
   number={3},
   pages={747--773},
}

\bib{sauer+thumann}{article}{
   author={Sauer, Roman},
   author={Thumann, Werner},
   title = {Topological Models of Finite Type for Tree Almost Automorphism Groups},
   journal = {International Mathematics Research Notices},  
   date = {2016}, 
   doi = {10.1093/imrn/rnw232}, 
}
%\bib{solomon}{article}{
%   author={Solomon, Louis},
%   title={The orders of the finite Chevalley groups},
%   journal={J. Algebra},
%   volume={3},
%   date={1966},
%   pages={376--393},
%   %issn={0021-8693},
%   %review={\MR{0199275}},
%   %doi={10.1016/0021-8693(66)90007-X},
%}
%		
%
	\bib{takesaki}{book}{
   author={Takesaki, M.},
   title={Theory of operator algebras. I},
   series={Encyclopaedia of Mathematical Sciences},
   volume={124},
   note={Reprint of the first (1979) edition;
   Operator Algebras and Non-commutative Geometry, 5},
   publisher={Springer-Verlag, Berlin},
   date={2002},
   pages={xx+415},
}
\bib{thom-rank}{article}{
   author={Thom, Andreas},
   title={$L^2$-invariants and rank metric},
   conference={
      title={$C^\ast$-algebras and elliptic theory II},
   },
   book={
      series={Trends Math.},
      publisher={Birkh\"auser, Basel},
   },
   date={2008},
   pages={267--280},
   %review={\MR{2408147}},
   %doi={10.1007/978-3-7643-8604-7_14},
}
\bib{MR2399103}{article}{
   author={Thom, Andreas},
   title={$L^2$-cohomology for von Neumann algebras},
   journal={Geom. Funct. Anal.},
   volume={18},
   date={2008},
   number={1},
   pages={251--270},
}

\bib{MR2486803}{article}{
   author={Thom, Andreas},
   title={Low degree bounded cohomology and $L^2$-invariants for
   negatively curved groups},
   journal={Groups Geom. Dyn.},
   volume={3},
   date={2009},
   number={2},
   pages={343--358},
}
\bib{tomdieck}{book}{
      author={tom Dieck, Tammo},
       title={Transformation groups},
      series={de Gruyter Studies in Mathematics},
   publisher={Walter de Gruyter \& Co., Berlin},
        date={1987},
      volume={8},
      
}

\end{biblist}
\end{bibdiv}
\end{document}